\documentclass[11pt]{article}
\usepackage[a4paper]{geometry}
\usepackage[utf8]{inputenc}
\usepackage[english]{babel}
\usepackage{amsthm,amsmath,amssymb}
\usepackage[usenames,dvipsnames,svgnames,table]{xcolor}
\usepackage[all]{xy}
\usepackage{xy,mathrsfs,pstricks,color,verbatim,numprint}
\usepackage{color,xcolor,transparent,subfigure,soul}  
\usepackage{tikz,ltablex,booktabs}
\usetikzlibrary{positioning,arrows.meta,decorations.text,calc}
\usetikzlibrary{shapes.multipart}

\usepackage{verbatim}

\input xy
\xyoption{all}

\newdir{ >}{{}*!/-5pt/\dir{>}}

\theoremstyle{definition} 
\newtheorem{theorem}{Theorem}[section]
\newtheorem{proposition}[theorem]{Proposition}
\newtheorem{lemma}[theorem]{Lemma}
\newtheorem{cor}[theorem]{Corollary}
\newtheorem{remark}[theorem]{Remark}
\newtheorem{example}[theorem]{Example}
\newtheorem{definition}[theorem]{Definition}

\newcommand{\A}{\mathcal{A}}
\newcommand{\B}{\mathcal{B}}
\newcommand{\E}{\mathcal{E}}
\newcommand{\cS}{\mathcal{S}}
\newcommand{\mN}{\mathbb{N}}
\newcommand{\N}{\mathbb{N}}

\def\Ext{\mbox{Ext}}
\def\Ex{\mbox{Ex}}
\def\irr{\mbox{irr}\,}
\def\rad{\mbox{rad}\,}
\def\rep{\mbox{rep}\,}
\def\add{\mbox{add}\,}
\def\mod{\mbox{mod}\,}
\def\dim{\mbox{dim}\,}
\def\coker{\mbox{coker}\,}

\newcommand{\imono}[1]{ \;\xymatrix{  \ar@{>->}^{#1}[r] &  \\} }
\newcommand{\iepi}[1]{ \;\xymatrix{  \ar@{->>}^{#1}[r] &  \\} }
\newcommand{\mono}{ \;\xymatrix{  \ar@{>->}[r] &  \\} }
\newcommand{\epi}{ \xymatrix{   \ar@{->>}[r] &  \\} }

\title{Reduction of exact structures}
\author{Thomas Br\"ustle, Souheila Hassoun, Denis Langford, Sunny Roy}
\begin{document}
\date{}
\maketitle

\abstract{ 

Examples of exact categories in representation theory are given by the category of $\Delta-$filtered modules over quasi-hereditary algebras, but also by various categories related to matrix problems, such as poset representations or representations of bocses.
Motivated by the matrix problem background, we study in this article the reduction of exact structures, and consider the poset $(Ex(\A), \subset)$ of all exact structures on a fixed additive category $\A$. This poset turns out to be a bounded complete lattice, and under suitable conditions results of Enomoto's imply that  it is boolean. 

We initiate in this article a detailed study of exact structures $\E$ by generalizing notions from abelian categories such as the length of an object relative to $\E$ and the quiver of an exact category $(\A,\E)$. We investigate the Gabriel-Roiter measure for $(\A,\E)$, and further study how these notions change when the exact structure varies.
}
\section{Introduction}

There are several notions of exact categories given by Barr, Buchsbaum or Quillen. We study in this article Quillen's \cite{Qu} notion of exact category, which is formulated in the context of an additive category $\A$. One specifies a distinguished class $\E$ of short exact sequences which forms an exact structure on $\A$, that is, $\E$ consists of kernel-cokernel pairs subject to some closure requirements, see section \ref{basic section}.
The pair 
$(\A,\E)$ is called an exact category  (we also refer to  \cite{GR} and \cite{Bu} for the system of axioms we are using).

It is well known that on every additive category $\mathcal{A}$ the class of all split exact sequences provides the smallest exact structure,  see \cite[Lemma 2.7]{Bu}.
However, for the maximal exact structure there is quite some recent literature, such as \cite{SW}, \cite{Cr}, \cite{Ru11} and \cite{Ru15} which shows that every additive category admits a unique maximal exact structure $\E_{max}$. We recall the details in section \ref{basic section}.

Quillen defined the abstract notion of exact structure somewhat as a by-product in his fundamental paper on higher algebraic K-theory. 
It allows to perform homological algebra relative to the exact structure $\E$, and to study the (relative) Grothendieck group and the derived category of $(\A,\E)$, see \cite{Bu}.
Relative homological algebra (like relative projective objects) has also been studied intensely from a different point of view, starting with a paper by Auslander and Solberg \cite{AS} where they look at subbifunctors of the Ext-functor. 
It has been shown in \cite{DRSS} that these two concepts coincide, that is, the additive closed subbifunctors correspond to exact structures. 
Recently, exact structures have become focus of work by several authors, like \cite{Enomoto} who classifies exact structures on a given Krull-Schmidt category of finite type, using Auslander algebra, or \cite{INP} where the more general concept of extriangulated structures is studied.
\medskip

While every  exact category $(\mathcal{A}, \mathcal{E})$ can be embedded into a module category, notions like length or simple object cannot be borrowed from such an embedding. 
The first goal of this paper is to give an intrinsic definition relative to the class of morphisms in $\mathcal{E}$, thus, in section \ref{section 3}, we call an object $\mathcal{E}-$simple if it does not admit proper monomorphisms that belong to the class $\mathcal{E}$. 
And we say that $X$ is an $\E-$suboject of $Y$, or $X \subset_\E Y$ if there exists a monomorphism in $\E$ from $X$ to $Y$. 
This change of definition requires to work out a number of notions and results that are granted in abelian categories, such as the notion of simple objects, artinian and finite objects or the length of an object.
It turns out that in general not all the desired properties can be guaranteed. 
We also define, in section \ref{section 3}, the notion of the quiver of an exact category $Q(\A, \E)$.

The motivation for studying reductions of exact structures stems from the matrix reduction technique. The method of matrix reduction has been applied successfully by the Kiev school to solve various important problems in representation theory, like the Brauer-Thrall conjectures, or to show the tame-wild dichotomy. 
While the basic technique is elementary, the formalism of matrix reductions is somewhat complicated. Various models have been proposed to formalise matrix reductions: poset representations or bimodule problems cover only some cases. For the general case, one needs to study bocs representations, as introduced by Roiter in \cite{Roi2}, or iterated quotients of bimodule problems as in \cite{Bru}. 

\medskip

No matter which formalism one chooses, the iterated application of reductions leads to more and more complicated categories. 
We propose a different approach in this paper, that is: {\em Keep the objects of the original category, but change its exact structure.} We illustrate in section  \ref{section:matrix reduction} with an example that the elementary technique of matrix reduction can be viewed as a reduction of exact structures,
where we define the {\em reduction} of an exact category $(\A, \E)$
as the choice of an exact structure $\E' \subseteq \E$. 
We observe that when $\E'$ is the smallest possible exact structure, the split exact structure, then the exact category $(\A, \E')$ is in some sense semisimple: Every indecomposable is simple. In general, $(\A, \E')$ will be "simpler" than $(\A, \E)$ in the sense that $(\A, \E')$ will have more simple objects.

We like to mention that the category of poset or bocs representations admits a natural exact structure, but these cases are rather special: the exact categories stemming from bocses always admit sufficiently many projectives, and are hereditary (the higher Ext groups vanish). 
The reduction of exact structures studied in this paper is therefore more general.
\color{black}
\medskip

A second goal of this paper is to study for a fixed additive category $\A$ the poset $(Ex(\A), \subset)$ of exact structures ordered by containment. It turns out that this poset is a complete bounded lattice, see section \ref{lattice}.

Another goal is to generalise the notion of Gabriel-Roiter measure to the realm of exact categories. To start, we first define in section \ref{section 6} the length of an object in an exact category $(\A,\E):$ 
the $\E-$length $l_{\E}(X)$ of an object $X$ is the maximal length of a chain of proper  $\E-$subojects of $X$.
We use the notion of $\E-$length to show the following result:

\begin{proposition}(see \ref{poset}):
Let $(\A, \E)$ be an essentially small exact category where every object has finite $\E-$length. Then
the relation $\subset_\E $ induces a partial order on $Obj\A$.
\end{proposition}
This result allows to show that the length function $l_{\mathcal{E}}$ of a finite essentially small exact category $(\A, \E)$ is a measure for the poset $Obj\mathcal{A}$ in the sense of Gabriel \cite{Gab}.
We further show that \emph{most} of the work of Krause \cite{Kr11} on the Gabriel-Roiter measure for abelian length categories can be generalized to the context of exact categories: 
For the partially ordered set $(ind\A,{\subset}_{\E})$ equipped with the length function $l_{\E}$, we define the Gabriel-Roiter measure as a morphism of partially ordered sets which refines the length function $l_{\E}$, see Theorem \ref{GR measure}:
\begin{theorem}
There exists a Gabriel-Roiter measure  for $ind(\A,\E)$.
\end{theorem}

\medskip

Finally, starting from the maximal exact structure on $\mathcal{A}$ one can choose a sequence of reductions to arrive at the minimal, the split exact structure. In \ref{section under reduction} we study these chains of exact structures in the lattice $Ex(\mathcal{A})$ and how basic notions, like the extended notion of length of an object, change under these reductions:

\begin{proposition} If $\E$ and $\E'$ are exact structures on $\A$, such that $\E' \subseteq \E$, then $l_{\E'}(X) \leq l_{\E}(X)$ for all objects $X$ of $\A$.
\end{proposition}
Thus, the length of objects is reduced, until the reduction reaches the minimum $\E_{min}$, where every indecomposable has length one, that is, simple. 
\bigskip

\paragraph{Acknowledgements.}
The authors would like to thank Charles Paquette for his thoughtful comments on the PhD thesis of the second author. The authors would also like to thank  Amit Shah, Min Huang, Emily Barnard, Bill Crawley-Boevey, Haruhisa Enomoto, Al Garver and Rose-Line Baillargeon for helpful discussions.\\ The authors were supported by Bishop's University, Université de Sherbrooke and NSERC of Canada.

\section{Definitions and basic properties}\label{basic section}

\subsection{Quillen exact categories}
We recall from \cite{GR,Bu} the definition of exact categories  in the sense of Quillen \cite{Qu} and give some examples.
\begin{definition}\label{Quillen def}Let $\mathcal{A}$ be an additive category. A kernel-cokernel pair $(i, d)$ in $\mathcal{A}$ is a pair of composable morphims such that $i$ is kernel of $d$ and $d$ is cokernel of $i$.
If a class $\mathcal{E}$ of kernel-cokernel pairs on $\mathcal{A}$ is fixed, an {\em admissible monic} is a morphism $i$ for which there exist a morphism $d$ such that $(i,d) \in \mathcal{E}$. An {\em admissible epic} is defined dually. Note that admissible monics and admissible epics are referred to as inflation and deflation in \cite{GR}, respectively. 
We depict an admissible monic by  $\mono$
and an admissible epic by $\epi$.
An {\em exact structure} $\mathcal{E}$ on $\A$ is a class of kernel-cokernel pairs $(i, d)$ in $\A$ which is closed under isomorphisms and satisfies the following axioms:
\medskip

\begin{tabular}{ll}
(E0) & For all objects  $A$ in $\mathcal{A}$ the identity $1_A$ is an admissible monic
\medskip\\
{(E0)$^{op}$} & For all objects  $A$ in $\mathcal{A}$ the identity $1_A$ is an admissible epic 
\medskip\\
(E1) & The class of admissible monics is closed under composition
\medskip\\
(E1)$^{op}$ & The class of admissible epics is closed under composition
\medskip\\
(E2) & \begin{minipage}[t]{10.5cm}
 The push-out of an admissible monic $i: A \mono B$ along an arbitrary morphism $t: A \to C$ exists and yields an admissible monic $s_C$:
$$\xymatrix{
A \; \ar[d]_{t} 
\ar@{ >->}[r]^{i}  \ar@{}[dr]|{\text{PO}} 
& B\ar[d]^{s_B}\\
C \; \ar@{>->}[r]_{s_C} & S}$$
\end{minipage}\medskip
\\
{(E2)}$^{op}$ &  \begin{minipage}[t]{10.5cm}
The pull-back of an admissible epic $h$ along an arbitrary morphism $t$ exists and yields an admissible epic $p_B$
$$\xymatrix{
P \; \ar[d]_{P_A} 
\ar@{ ->>}[r]^{p_B} \ar@{}[dr]|{\text{PB}}  & B\ar[d]^{t}\\
A \; \ar@{->>}[r]_{h} & C}$$
\end{minipage}
\end{tabular}

An {\em exact category} is a pair $(\mathcal{A}, \mathcal{E})$ consisting of an additive category $\mathcal{A}$ and an exact structure $\mathcal{E}$ on $\mathcal{A}$. Elements of $\mathcal{E}$ are called short exact sequences.
Note that $\mathcal{E}$ is an exact structure on $\mathcal{A}$ if and only if $\mathcal{E}^{op}$ is an exact structure on $\mathcal{A}^{op}$.
For a fixed additive category $\A$, we denote by $\Ex(\A)$ the poset of exact structures $\E$ on $\A$, with order relation given by containment. In fact, $\Ex(\A)$ is a lattice, see section \ref{lattice}.
\end{definition}

\begin{example} \cite[Lemma 10.20]{Bu} Let $(\A, \E)$ be an exact category and $\mathcal{B}$ a full subcategory which is closed under extensions, that is, for every short exact sequence $$X\mono Y\epi Z$$ 
in $\E$ the object $Y$ belongs to $\mathcal{B}$ if the endterms $X$ and $Z$ are objects of $\mathcal{B}$. Then the pairs of $\E$ with components in $\mathcal{B}$ form an exact structure on $\mathcal{B}$.
For example a torsion class of an abelian category forms an exact category since it is an extension closed subcategory.
\end{example}

\subsection{Types of additive categories}
Certain properties of the underlying additive category $\A$  determine which exact structures can exist on $\A$. We recall here the definition of various types of additive categories, and of some classes of short exact sequences. We then discuss in \ref{minimal} and \ref{maximal} some consequences on the existence of exact structures.

We begin with a large class of additive categories, the {\em weakly idempotent complete} categories:
\begin{definition}
Following \cite{Bu}, we call an additive category $\A$ {\em weakly idempotent complete} (w.i.c.) if all retractions have kernels and all sections have cokernels.
In fact, B\"uhler shows in \cite[Lemma 7.1]{Bu} that it is sufficient to have one of the two conditions. Moreover, in \cite[Corollary 7.5]{Bu} it is shown that $\A$ is weakly idempotent complete if and only if every retraction is an admissible epic for all exact structures on $\A$, and dually,  every section is an admissible monic for all exact structures on $\A$.
\end{definition}
\begin{definition}
An additive category $\A$ is \emph{idempotent complete} (i.c) if every morphism $e : X \rightarrow  X$ in $\A$ satisfying $e^{2} = e$ has a kernel, or equivalently, a cokernel. 
\end{definition}

\begin{definition} An additive category $\A$ is \emph{semi-abelian} if it is \emph{pre-abelian} (has kernels and cokernels) and 
the induced canonical map 
$$\bar{f}: Coimf \rightarrow Imf$$ is a bimorphism, i.e, a monomorphism and an epimorphism.
\end{definition}

\begin{definition}\cite[p.524]{RW77} A kernel $(A,f)$ is called \emph{semi-stable} if for every push-out square 
\[\xymatrix{
A \; \ar[d]_{t} 
\ar@{ ->}[r]^{f}  \ar@{}[dr]|{\text{PO}} 
& B\ar[d]^{s_B}\\
C \; \ar@{->}[r]_{s_C} & S}
\]

\noindent the morphism $s_C$ is also a kernel.
We define dually a \emph{semi-stable} cokernel.
A short exact sequence $\xymatrix{
A \, \ar@{ >->}[r]^{i} & B \ar@{ ->>}[r]^{d} & C
}$ is said to be \emph{stable} if
$i$ is a semi-stable kernel and $d$ is a semi-stable cokernel. We denote by $\E_{sta}$ the class of all \emph{stable} short exact sequences.
\end{definition}

\begin{definition}A morphism $f$ is called \emph{strict} if the canonical map 
$\bar{f}$ is an isomorphism.
A short exact sequence $\xymatrix{
A \, \ar@{ >->}[r]^{i} & B \ar@{ ->>}[r]^{d} & C
}$ is said to be \emph{strict} if
$i$ is strict or $d$ is strict. 
We denote by $\E_{str}$ the class of all strict short exact sequences.
\end{definition}

\begin{definition}An additive category $\A$ is \emph{quasi-abelian} if it is \emph{pre-abelian} and all kernels and cokernels are \emph{semi-stable}.

Moreover, an additive category $\A$ is \emph{quasi-abelian} if it is \emph{pre-abelian} and every pullback of a strict epimorphism is a strict epimorphism, and every pushout of a strict monomorphism is a strict monomorphism.
\end{definition}

\begin{definition}An additive category $\A$ is \emph{abelian} if it is \emph{pre-abelian} and all morphisms are \emph{strict}.
\end{definition}
\begin{remark}A \emph{pre-abelian} category admits pullbacks and pushouts.
\end{remark}
\begin{remark}The hierarchy of  additive categories which we discussed here is as follows (where all inclusions are strict):
\end{remark}
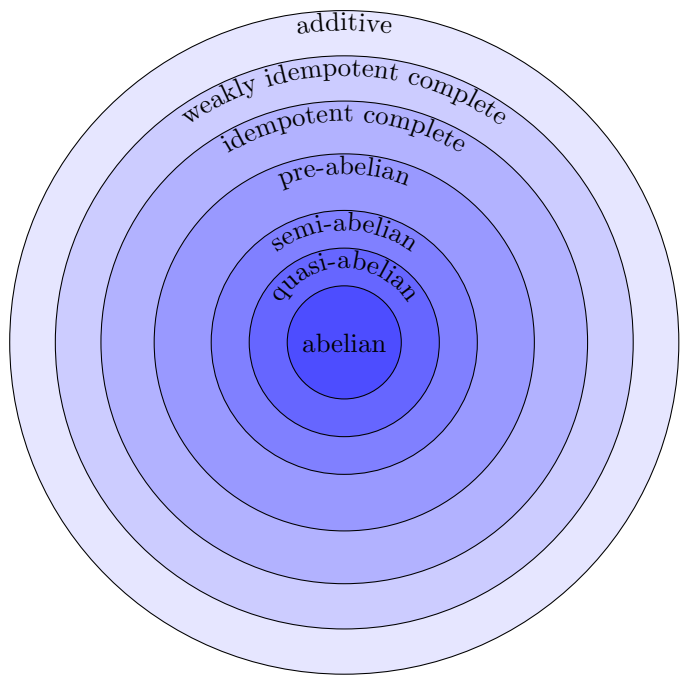
\begin{center}
    \begin{tikzpicture}
  [
    circle type 1/.style={draw=none,align=center, font=\scriptsize,text width = 1.25cm},
  ]
  \coordinate (c2) at (0,0);
  \newlength{\smallercircle}
  \foreach \i / \j [count=\ino] in {4.4cm/additive, 
  3.8cm/weakly idempotent complete, 3.2cm/idempotent complete, 2.6cm/pre-abelian, 2.0cm/semi-abelian, 1.4cm/quasi-abelian,
  0.75cm}
    {
     \pgfmathsetmacro{\colmixer}{mod(10*\ino,100)}%
     \path [draw, fill=blue!\colmixer] (c2) ++(0,-\i) arc (-90:270:\i);
     \ifnum\ino<7
       \setlength{\smallercircle}{\i}
       \addtolength{\smallercircle}{-7.5pt}
       \path [decoration={text along path, text={\j}, text align=center, reverse path}, decorate] (c2) ++(0,-\smallercircle) arc (-90:270:\smallercircle);
     \fi
    }
\filldraw[black] (0,0) circle (0pt) node[anchor=center] {abelian};
\end{tikzpicture}
\end{center}
\subsection{The minimum exact structure}\label{minimal}
It is well known that every additive category admits a unique minimal exact structure $\E_{min}$:
\begin{proposition}\label{split}
\cite[example 13.1]{Bu}
For every additive category $\mathcal{A}$ the sequences isomorphic to 
\[\xymatrix{
A \, \ar@{ >->}[r]^{\footnotesize\begin{bmatrix}
   1 \\
    0 \\
\end{bmatrix}
} & A\oplus B \ar@{ ->>}[r]^{[0 1]} & B \\
}
\]
form an exact structure $\mathcal{E}_{min}$, called the split exact structure. 
In fact, every exact structure on $\A$ contains all split exact sequences \cite[Lemma 2.7]{Bu}, which makes $\mathcal{E}_{min}$ the minimum in the lattice $\Ex(\A)$ of all exact structures on $\A$.

If $\mathcal{A}$ is weakly idempotent complete, then each retraction can be completed to a split sequence as above, hence the exact structure $\mathcal{E}_{min}$ is formed by all pairs $(s,r)$ of sections with retractions. 
\end{proposition}

\begin{example}\label{submonoid}
Let $S \subset (\mN,+)$ be a submonoid, that is, $S$ is an additively closed set containing zero. Consider the category $\mathcal{A}_S$ of vector spaces $V$ over a field $k$ of dimension $\dim_k V \in S$. For a short exact sequence in $\mod k$
$$0 \to U \to V \to W \to 0$$
we have that $U$ and $W$ in $\A_S$ implies $V$ in $\A_S$ since $S$ is additively closed.
Thus $\A_S$ is additive since it is an extension-closed full subcategory of the additive category $\mod\, k$, and by example \ref{split}, the split exact sequences in $\A_S$ form an exact structure $\E_{min}$ on $\A_S$. 
Note that $\A_S$ is not weakly idempotent complete when $S \neq \mN$ since there are retractions whose kernel is not in $\A_S$.
\end{example}

\subsection{The maximum exact structure}\label{maximal}
It is a deeper result that every additive category also admits a unique maximal exact structure $\E_{max}$. We review some of the recent literature on this subject:
\begin{theorem}\label{Emax}\cite[Corollary 2]{Ru11} Every additive category admits a unique maximal exact structure $\E_{max}$.
\end{theorem}

The drawback of this result is that an explicit description of the maximum exact structure is not known in general. However, for certain types of additive categories, the exact structure $\E_{max}$ can be described explicitly. The following theorem generalizes the result on pre-abelian categories from \cite[Theorem 3.3]{SW}:

\begin{theorem}\cite[Theorem 3.5]{Cr} Let $\A$ be an weakly idempotent complete category. Then the stable exact sequences $\E_{sta}$ define an exact structure on $\A$. Moreover, this is the maximal exact structure $\E_{max}$ on $\A$.
\end{theorem}

\begin{remark}The short exact sequences forming the maximal exact structure $\E_{max}$ do not always coincide with the stable short exact sequences in $\E_{sta}$. In fact we have that $\E_{max}\subseteq \E_{sta}$, so in case the class $\E_{sta}$ forms an exact structure it will be the maximal one.
See \cite{Ru15} for an example where $\E_{max}\subsetneq \E_{sta}$.
\end{remark}

\begin{theorem}\cite[1.1.7]{Sch}(\cite[4.4]{Bu}) In any quasi-abelian category, the class of all short exact sequences defines an exact structure $\E_{all}$ and this is the maximal one $\E_{max}=\E_{all}$. In particular this is the case for abelian categories (see also \cite{Ru01}).
\end{theorem}

\begin{remark}The class of all short exact sequences $\E_{all}$ does not necessarily form an exact structure for any additive category since pushouts of kernels need not be kernels. For a counter-example, take the category of abelian $p-$groups with no elements of infinite height, see \cite[page 522]{RW77}.
But in case $\E_{all}$ forms an exact structure, it will be the maximal one.
\end{remark}

\subsection{More examples}
\begin{example}\label{triangulated}
If  $\A$ is a triangulated category then every monomorphism splits, and so $\E_{max}=\E_{sta}=\E_{min}$ forms the only possible exact structure on $\A$.
\end{example}
\begin{example}\cite{Qu} A quasi-abelian category $\A$ together with $\E_{str}$ is an exact category $(\A, \E_{str})$. See also \cite[section 4]{Bri}.
\end{example}
\begin{example}Every full subcategory of an abelian category which is closed under direct sums and direct summands is idempotent complete and $\E_{max}=\E_{sta}$.
\end{example}

\begin{example}\cite{Ru08} Let $A=kQ/I$ be the path algebra over a field $k$ given by the following quiver $Q$ with relations $I$ generated by commutativity relations at the two squares (note that the algebra $A$ is tilted of type $E_6$):

\[
\xymatrix{
1 \ar@{ ->}[r] \ar@{->}[d]
& 2 \ar@{->}[d] & 3 \ar@{ ->}[l] \ar@{->}[d] & \\
4\ar@{->}[r] & 5 & 6 \ar@{->}[l]}
\]
 We consider the category $\A=A-proj$ of finitely generated projective $A-$modules. 
This  $\A$ was the first example of a semi-abelian category which is not quasi-abelian.
In particular, $\A$ is weakly idempotent complete
and $\E_{max}=\E_{sta}$. 
\end{example}

\begin{example}
Let $\A$ be the category of Banach spaces or Fr\'echet spaces, then $\A$ is quasi-abelian and $\E_{max}=\E_{all}$.
\end{example}
\section{The quiver of an exact category}\label{section 3}

\subsection{ $\E-$subobjects and $\E-$simple objects}

Throughout this section, let $\mathcal{E}$ be an exact structure for an additive category $\A$.
We define the notion of $\E-$subobjects and $\E-$simple objects.
\begin{definition}Let $A$ and $B$ be objects of $(\mathcal{A},\E)$. 
We write $A {\subset}_{\mathcal{E}} B $ and say that $A$ is an {\em admissible subobject} or {\em $\mathcal{E}-$subobject of $B$}, if there is an admissible monic $A \imono{i} B$ from $A$ to $B$. 
If in addition $i$ is not an isomorphism, we use the notation  $A {\subsetneq}_{\mathcal{E}} B $ and say that $A$ is a \emph{proper} admissible subobject of $B$.
\end{definition}

\begin{remark}\label{zero coker}
An admissible monic $A \imono{i} B$ is proper precisely when its cokernel is non-zero. 
In fact, by uniqueness of kernels and cokernels, the exact sequence 
$$B\imono{1_B} B\epi 0$$ 
is, up to isomorphism, the only one with zero cokernel. Thus an admissible monic $i$ has $\coker i = 0$ precisely when $i$ is an isomorphism. Dually, an admissible epic $B \iepi{d} C$ is an isomorphism precisely when $\ker d = 0$.
\end{remark}

\begin{definition}
A non-zero object $S $ in $(\A,\E)$ is {\em $\mathcal{E}-$simple} if $S$ admits no $\E-$sub\-objects except $0$ and $S$, that is, whenever $ A \subset_\E S$, then $A$ is the zero object or isomorphic to  $S$.
\end{definition}
\begin{remark}\label{E-simple indecomposable}
An $\mathcal{E}-$simple object is indecomposable, since the canonical inclusion 
$X_1 \imono{i_1} X_1 \oplus X_2$
is admissible in every exact structure, see example \ref{split}.
Conversely, when  $\mathcal{E}$ is the split exact structure from example \ref{split}, then every indecomposable object is $\mathcal{E}-$simple. 
\end{remark}

\begin{example}\label{submonoid2}
Consider the category $\A_S$ of vector spaces from example \ref{submonoid} for the monoid $S = \N \backslash \{ 1 \}$, equipped with the split exact structure $\E=\E_{min}$.
Then the $\E-$simple objects in $\A_S$ are $k^2$ and $k^3$, up to isomorphism.
This corresponds to the fact that the monoid $S$ admits $\{2,3\}$ as minimal generating set.
\end{example}

\subsection{The quiver of $(\A,\E)$}\label{quiver section}

The aim of this section is to define the quiver of an exact category, and compare it with different notions studied in the literature. We assume here that $\A$ is not only additive, but a $k-$category for some field $k$. It is shown in \cite{DRSS} that the datum of an exact structure $\E$ on $\A$ corresponds to the choice of an additive bifunctor 
$$\Ext_\E(-,-) : \A^{op} \times \A \to Mod_k$$ which is closed in the sense of M.C.R. Butler and G. Horrocks \cite{BH}.
Here $\Ext_\E(Z,X)$ denotes the set of all exact pairs 

$$X\imono{i} Y\iepi{d} Z$$  in $\E$ modulo the usual equivalence relation of short exact sequences, which turns into a vector space under Baer sum.

\begin{definition}\label{rad-def}
We recall the (Jacobson) radical of an additive $k$-category and its powers from \cite{ASS}:
\begin{enumerate}
\item[-] The radical $\rad_\A$ of an additive $k$-category $\A$  is the two-sided ideal given for all objects X and Y in $\A$ by the $k$-vector space $\rad_\A(X,Y)$ formed by all $f\in\A(X,Y)$ such that $1_X-g\circ f$ is invertible for all $g\in\A(Y,X)$.
\item[-] Given $m\geq1$, the $m^{th}$  power  $ rad_{\A}^{m}\subseteq rad_{\A}$ of $rad_{\A}$ is obtained by taking for $rad_{\A}^{m}(X,Y)$ the subspace of $rad_{\A}(X,Y)$ containing all finite sums of morphisms of the form
$$\xymatrixcolsep{2pc}\xymatrix{
X=X_{1}\ar[r]^{f_1}&X_{2}\ar[r]^{f_2}&\dots\ar[r]^{f_{m-1}}&X_{m-1}\ar[r]^{f_{m}}&X_m=Y  \; }$$
where $f_i:X_{i-1}\rightarrow X_{i} \in rad_{\A}(X_{i-1},X_i)$ for all i=1,\dots,m.
\end{enumerate}
\end{definition}


\begin{definition}\label{quiver-def}Let $\A$ be a small skeletally small and Hom-finite additive category. The \emph{quiver $Q(\A,\E)$ of the exact category $(\A,\E)$} is the graded quiver given as follows:
\begin{enumerate}
\item[-] the vertex set $Q_0(\A,\E)$ is the set of isomorphism classes of $\E-$simple objects.
\end{enumerate}
For two vertices represented by $\E-$simple objects $X$ and $Y$, we further define:  
\begin{enumerate}
\item[-] the number of arrows of degree zero from $X$ to $Y$ equals  the dimension of the space $\irr(X,Y)=\rad_\A(X,Y) / \rad^2_\A(X,Y)$ of irreducible morphisms in $\A$ from $X$ to $Y$.
\item[-] the number of arrows of degree one from $X$ to $Y$ equals  the dimension of the vector space $\Ext_\E(X,Y).$
\end{enumerate}
We draw in illustrations the arrows of degree zero by dotted lines, and the arrows of degree one by solid lines.
\end{definition}

\begin{example}
Let $k$ be an algebraically closed field, $A$ be an artinian $k-$algebra, and $\A = \mod A$. When $\E = \E_{all}$ is the maximal exact structure $\E_{max}$, then the quiver $Q(\A,\E_{max})$ is the ordinary (Gabriel) quiver of the algebra $A$, with all arrows of degree one.
For the minimal exact structure $\E = \E_{min}$, the simples are the indecomposable $A-$modules by \ref{E-simple indecomposable}, and the 
quiver $Q(\A,\E_{min})$ is the Auslander-Reiten quiver of $A$, with all arrows of degree zero. 
We will discuss in section \ref{section:matrix reduction} how reduction of exact structures transforms iteratively the Gabriel quiver into the Auslander-Reiten quiver of an algebra.
\end{example}

\begin{example}\label{posetrep}
The technique of matrix reduction has been studied using various models, such as representations of posets, subspace categories, bimodules or bocses.  We recall here from \cite[2.3 example 6]{GR} one example, the representations of posets (see also the books \cite{Ri84,Si}):
Given a poset $(S, \le)$, the category $\rep S$ of representations of $S$ is formed by matrices whose columns are subdivided into blocks corresponding to the elements $\{s_1, \ldots,s_n\}$ of $S$.
More formally, the objects of $\rep S$ are pairs $(d,M)$ where $d \in \N^{n+1}$, and $M$ is a matrix with entries in $k$ that has $d_0$ rows and $d_1+\cdots +d_n$ columns, subdivided into $n$ blocks of size $d_1,\ldots,d_n$:
$$M= \left[ \begin{array}{c|c|c}
    M_1 & \cdots & M_n  
  \end{array}\right]$$

A morphism $(d,M) \to (d',M')$ is given by a pair  of matrices $(X,Y)$ such that $XM=M'Y$ and where the matrix $Y$ has a block structure determined by the order relation in $S$, allowing operations from columns in block $i$ to columns in block $j$ only if $s_i \le s_j $ in $S$.
One could equivalently define an element in rep S as a couple (V,M) where M is the same matrix as defined earlier and V is a set of $n+1$ $k$-vector spaces $\{V_{0},V_{1},...,V_{n}\}$ of dimensions $\{d_{0},d_{1},...,d_{n}\}$ respectively. Thus, a morphism can be illustrated with the following commutative diagram:
$$\xymatrixcolsep{4pc}\xymatrix{
V_{0} \; \ar[d]_{X} 
& V_{1}\times...\times V_{n}\ar[d]^{Y} \ar@{->}[l]^{M\qquad}\\
V_{0}' \;  & V_{1}'\times ...\times V_{n}'\ar@{->}[l]_{M'\qquad}}$$

As in \cite[section 9.1, example 5]{GR} (see also \cite[4.2]{DRSS} or \cite[2.3]{BrHi}), we equip the $k-$category $\A = \rep S$ with the exact structure $\E$ whose admissible monics are formed by morphisms $(X,Y)$ where both $X$ and $Y$ are sections. 

Let us consider (V,M) in rep S where $d_{0}\geq 1$ and the following morphism:
$$\xymatrixcolsep{4pc}\xymatrix{
s_0: & k \; \ar[d] 
&0\times...\times 0 \ar[d] \ar@{->}[l]\\
 & V_{0} \;  & V_{1}\times ...\times V_{n}\ar@{->}[l]_{M\quad}}$$
The object $s_0:= (\{k,0,...,0\},0)$ is a 
simple $\E-$subobject of $(V,M)$. In fact $s_0$ is the unique simple object having $d_{0}\geq 1$. Let us now fix $d_{0}=0$ and $(V,M)\neq(0,0)$. Suppose $d_{i} \geq 1$ for a certain $i$, then the following morphism gives  an $\E$-subobject $s_{i}$ of (V,M):
$$\xymatrixcolsep{4pc}\xymatrix{
s_i: & 0 \; \ar[d] 
 &0\times...\times k\times...\times 0 \ar[d]^{q_{i}} \ar@{->}[l]\\
 & 0 \;  & V_{1}\times ...\times k^{d_{i}}\times...\times V_{n}\ar@{->}[l]}$$

where $s_{i} = (V',0)$ with $V_{i}'=k$ and all other spaces zero. This shows that the set $\{s_0,s_{1},...,s_{n}\}$ gives all $\E$-simple objects in rep S. 
There is a non-zero morphism $f_{ij}$ from $s_i$ to $s_j$ whenever  $s_i \le s_j$ in the poset $S$. Note that each of these morphisms $f_{ij}$ is a monomorphism and an epimorphism in the category $\rep S$, but for $i \neq j$ these are not isomorphisms, and not admissible monics nor admissible epics.

Furthermore, the following family of short exact sequences determines arrows of degree one in the quiver $Q(\A,\E)$: 
$$\xymatrixcolsep{4pc}\xymatrix{
s_i: & 0 \; \ar[d] 
&0\times...\times k\times...\times 0 \ar[d]^{q_{i}} \ar@{->}[l]\\
 & k \ar@{->}^{1}[d]\;  & 0\times ...\times k\times...\times 0\ar@{->}[d]\ar@{->}_{1\qquad\quad }[l]\\
s_0: & k & 0\times...\times 0 \ar@{->}[l]}$$
In fact, one can verify that $\dim\Ext_\E(s_0,s_i)=1$ and that there are no other extensions between these objects. The quiver $Q(\A,\E)$ is therefore formed by the Hasse quiver of $S$ with arrows of degree zero, together with an extra vertex $s_0$ that sends an arrow of degree one to each vertex of $S$.
We illustrate below an example of the quiver $Q(\A,\E)$ for a poset $S$ with Hasse diagram $H(S)$:
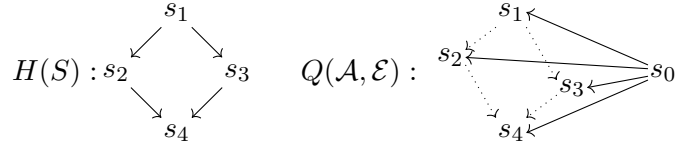
\begin{center}
\begin{figure}[h]
\centering
\begin{tikzpicture}[scale=0.4]
	\centering
   \node[draw=none] at (-4,-2){$H(S):$};
   \node[draw=none] at (0,0){$s_1$};
   \node[draw=none] at (2,-2){$s_3$}; 
   \node[draw=none] at (-2,-2){$s_2$}; 
   \node[draw=none] at (0,-4){$s_4$}; 
   \node[draw=none] at (6,-2){$Q(\A,\E):$};
   \node[draw=none] at (11,0){$s_{1}$};
   \node[draw=none] at (13,-2.5){$s_{3}$}; 
   \node[draw=none] at (9,-1.5){$s_{2}$}; 
   \node[draw=none] at (11,-4){$s_{4}$};
   \node[draw=none] at (16,-2){$s_0$};
   
   \draw[->](0.5,-0.5)--(1.5,-1.5);
   \draw[->](-0.5,-0.5)--(-1.5,-1.5);
   \draw[->](1.5,-2.5)--(0.5,-3.5);
   \draw[->](-1.5,-2.5)--(-0.5,-3.5);
   \draw[->,dotted](11.5,-0.5)--(12.5,-2.25);
   \draw[->,dotted](10.5,-0.5)--(9.5,-1.25);
   \draw[->,dotted](12.5,-2.75)--(11.5,-3.5);
   \draw[->,dotted](9.5,-1.75)--(10.5,-3.5);
   \draw[<-](11.5,0)--(15.5,-1.75);
   \draw[<-](11.5,-4)--(15.5,-2.25);
   \draw[<-](13.5,-2.5)--(15.5,-2.125);
   \draw[<-](9.5,-1.5)--(15.5,-1.875);
   
\end{tikzpicture}
\end{figure}
\end{center} 
\end{example}

\section{Matrix reduction versus exact structures}\label{section:matrix reduction}

The aim of this section is to link matrix reduction to reduction  of exact structures.
We first present in \ref{subsection:matrix reduction} as an example a chain of matrix reductions, and illustrate some intermediate steps by certain quivers with dashed and solid arrows.
We then justify these pictures in \ref{subsection:exact structure reduction} and \ref{subsection:new exact structure}, showing that they are in fact the quivers of certain exact categories corresponding precisely to the intermediate steps of matrix reductions.

\begin{definition}\label{def reduction}A reduction of an exact category $(\A, \E)$ is the choice of an exact structure $\E' \subseteq \E$ giving rise to a new exact category $(\A, \E')$.
Here  we mean by $\E' \subseteq \E$  that every exact pair $(i, d)\in \E'$ also belongs to $\E$.
\end{definition}

\subsection{Matrix reduction}\label{subsection:matrix reduction}

We describe here an example of a matrix reduction, and later compare it to  reduction of exact structures. The matrix reduction is discussed in \cite[1.2]{GR}, we refer to some background there.

Reduction for a quiver of type $A_3$ is also discussed in example 4.56 in \cite{Ku}, where the bocs point of view is given, compare the biquivers shown there.
\color{black}
Consider the category $\A = \rep Q$ of representations of  the quiver
\[ Q : \qquad \xymatrix{ 1 \ar[r]^{\alpha}  & 2  & 3 \ar[l]_{\beta}}\]
The category $\A$ is equivalent to the category $ \rep S$ of representations of  the poset
$S=\{1,3\}$ of two incomparable elements.
As in example \ref{posetrep}, its objects $(d,M)$ are given by pairs of matrices $A$ and $B$ with the same number of rows, we write it as follows:
$$M=[A|B]$$
 The algebra $\Lambda$ operating on representations of dimension vector $d=(d_1,d_2,d_3)$ is given by pairs of square matrices $(X,Y)$ where $X \in k^{d_2 \times d_2}$ and
$$Y = \left[\begin{array}{cc}
     Y_1 & 0  \\
     0 & Y_3\\
\end{array}\right].$$
Note that $\Lambda$ is semisimple, and the quiver of  $\Lambda$ has three isolated vertices corresponding to the three simple representations $S_1,S_2,S_3$ of $Q$. The arrows of $Q$ describe extensions between these simple objects, corresponding to the two matrices $A$ and $B$.
\medskip

{\bf (1)} We choose in the first reduction step to transform the matrix $A$ into its normal form:
\[\xymatrix{
{\begin{bmatrix}
     A | B \\
\end{bmatrix}
}\ar@[red][r] & 
{\left[\begin{array}{cc|c}
     0 & 1 & B' \\
     0 & 0 & B''\\
\end{array}\right]  }
} \]
Here we denote by $1$ the identity matrix of size rank $A$. We now restrict the operating algebra to the subalgebra $\Lambda_1 \subset \Lambda$ formed by those pairs $(X,Y)$ that preserve the normal form on the matrix $A$, that is, 
$$X \left[\begin{array}{cc}
     0 & 1  \\
     0 & 0 \\
\end{array}\right]  = \left[\begin{array}{cc}
     0 & 1  \\
     0 & 0 \\
\end{array}\right]
\left[\begin{array}{cc}
     Y_1 & 0  \\
     0 & Y_3\\
\end{array}\right].
$$
Thus the matrix $X$ is replaced by a matrix 
\
$$X' = \left[\begin{array}{cc}
     X'_1 & X_{12}  \\
     0 & X''_1\\
\end{array}\right]$$

which induces a subdivision of the rows into two blocks labeled $2'$ and  $2''$.
The quiver of the algebra $\Lambda_1$ turns out to be the following: 

\[  \xymatrix{ 2'' \ar@{.>}[r]  & 2' \ar@{.>}[r]  & 1  & 3}\]
The yet unreduced part of the matrix, given by the two blocks $B'$ and $B''$, corresponds to extensions from $3$ to the row-blocks $2'$ and $2''$. We might visualize this by introducing two solid arrows:

\[  \xymatrix{  & 3 \ar[d]^{\beta'} \ar[dl]_{\beta''} & \\ 2'' \ar@{.>}[r]  & 2' \ar@{.>}[r]  & 1 }\]
\bigskip

{\bf (2)} In the second reduction step, we  transform the matrix $B''$ into its normal form, and use row transformations to produce a zero block above the newly created identity matrix in the part corresponding to matrix $B$:

\[\xymatrix{
{\left[\begin{array}{cc|c}
     0 & 1 & B' \\
     0 & 0 & B''\\
\end{array}\right]  }
\ar@[blue][r] & 
{\left[\begin{array}{cc|cc}
     0 & 1 & B_3 & 0 \\
     0 & 0 & 0 & 1\\
     0 & 0 & 0 & 0\\
\end{array}\right]  }
} \]
\bigskip

{\bf (3)} In the final reduction step, we  transform the matrix $B_3$ to normal form:

\[\xymatrix{
{\left[\begin{array}{cc|cc}
     0 & 1 & B_3 & 0 \\
     0 & 0 & 0 & 1\\
     0 & 0 & 0 & 0\\
\end{array}\right]  }
\ar@[green][r] & 
{\left[\begin{array}{ccc|ccc}
     0 & 1 & 0 & 0 & 1 & 0 \\
     0 & 0 & 1 & 0 & 0 & 0 \\
     0 & 0 & 0 & 0 & 0 & 1 \\
     0 & 0 & 0 & 0 & 0 & 0 \\
\end{array}\right]  }
} \]

At this moment, the representation $(d,M)$ is completely decomposed into a direct sum of indecomposable representations. The algebra $\Lambda_3$ of transformations preserving this decomposition is Morita-equivalent to the Auslander algebra of $Q$, hence the quiver of $\Lambda_3$ is the Auslander-Reiten quiver of $Q$.

\subsection{Reduction of exact structures}\label{subsection:exact structure reduction}
To illustrate the change of exact structures, we consider in this section one example of an additive category $\A$ and describe all its exact structures. 
We recall from \cite[2.10]{Bu} that an exact structure $\E$ can be viewed as an additive subcategory of the category Ch$(\A)$ of chain complexes of objects in $\A$. 
Thus, we can talk about indecomposable short exact sequences, and use notation like the direct sum $e \oplus e'$ of short exact sequences in $\E$, or add$(e)$ denoting the additive subcategory of $\E$ generated by the short exact sequence $(e)$.
\bigskip

We reconsider now the following example
\begin{example}\label{cube}
Consider the category $\A = \rep Q$ of representations of  the quiver
\[ Q : \qquad \xymatrix{ 1 \ar[r]^{\alpha}  & 2  & 3 \ar[l]_{\beta}}\]

The Auslander-Reiten quiver of $\A$ is as follows:
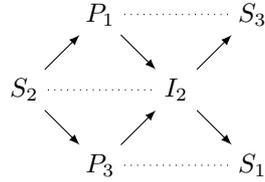
\begin{center}
\begin{tikzpicture}

\node (a) at (-1,0) {$S_2$};
\node (b) at (0,1) {$P_1$};
\node (c) at (0,-1) {$P_3$};
\node (d) at (1,0) {$I_2$};
\node (e) at (2,1) {$S_3$};
\node (f) at (2,-1) {$S_1$};
 \draw [-latex] (a) -- (b);
 \draw [-latex] (a) -- (c);
 \draw [-latex] (b) -- (d);
 \draw [-latex] (c) -- (d);
 \draw [-latex] (d) -- (e);
 \draw [-latex] (d) -- (f);
 \draw [dotted,-] (b) -- (e);
 \draw [dotted,-] (c) -- (f);
 \draw [dotted,-] (a) -- (d);
 \end{tikzpicture}
 \end{center}

We consider the following indecomposable non-split exact sequences where the first three are the Auslander-Reiten sequences:

\begin{enumerate}
\item[(AR1)] \hspace{2.5cm}$\xymatrix{ 0 \ar[r] & P_1 \ar[r]& I_2 \ar[r] & S_3  \ar[r] & 0}$
\item[(AR2)] \hspace{2.5cm}$ \xymatrix{ 0 \ar[r] & P_3 \ar[r]& I_2 \ar[r] & S_1 \ar[r]  & 0 }$
\item[(AR3)] \hspace{2cm}$ \xymatrix{ 0 \ar[r] & S_2 \ar[r] & P_1\oplus P_3 \ar[r] & I_2 \ar[r]  & 0 }$
\item[(4)] \hspace{2.5cm}$ \xymatrix{ 0 \ar[r] & S_2 \ar[r]& P_1 \ar[r] & S_1  \ar[r] & 0}$
\item[(5)] \hspace{2.5cm}$ \xymatrix{ 0 \ar[r] & S_2 \ar[r]& P_3 \ar[r] & S_3  \ar[r] & 0}$
\end{enumerate}
\bigskip

The following list enumerates \emph{all} exact structures $\E$ on $\A$:
\begin{enumerate}
 
  \item[$\bullet$]$\E_{min}$ is the class of all split short exact sequences, 
   \item[$\bullet$]${\E_{1}} =  \{X\oplus Y| X \in \E_{min} , Y \in \add(AR1)\}$,
   \item[$\bullet$]${\E_{2}} =  \{X\oplus Y| X \in \E_{min} , Y \in \add(AR2)\}$,

 \item[$\bullet$]${\E_{3}} =  \{X\oplus Y| X \in \E_{min} , Y \in \add(AR3)\}$,

\item[$\bullet$]${\E_{1,2}} = \{X\oplus Y\ | X \in \E_{1} , Y \in \E_2\}$,
 \item[$\bullet$]${\E_{1,3,5}} =  \{X\oplus Y\oplus Z | X \in \E_{1} , Y \in \E_{3} , Z \in \add(5)\}$,
 \item[$\bullet$]${\E_{2,3,4}} =  \{X\oplus Y\oplus Z | X \in \E_{2} , Y \in \E_{3} , Z \in \add(4)\}$,
 \item[$\bullet$]$\E_{max}$ is the class of all short exact sequences in $\A$. 

 \end{enumerate}
 \bigskip
 
 We have exactly $2^3 = 8$  exact structures since $\A$ admits 3 Auslander-Reiten sequences, and every exact structure is uniquely determined by a choice of a set of Auslander-Reiten sequences, see \ref{Grenouille verte}.
 
 \bigskip
 
Hence the lattice of exact structures $Ex({\A})$ is a cube, where the oriented arrows present inclusions:

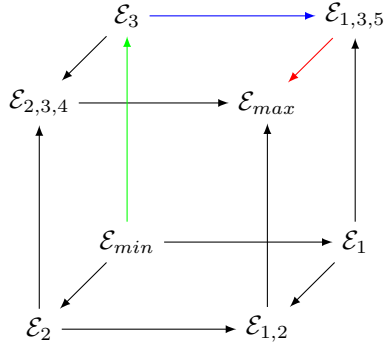
\begin{center}
\begin{tikzpicture}[scale=3]

 \node (a) at (0,0,0) {$\E_{min}$};

 \node (b) at (1,0,0) {$\E_{1}$};
 \node (c) at (1,1,0) {$\E_{1,3,5}$};
 \node (d) at (0,1,0) {$\E_{3}$};

 \node (e) at (0,0,1) {$\E_{2}$};
 \node (f) at (1,0,1) {$\E_{1,2}$};
 \node (g) at (1,1,1) {$\E_{max}$};
 \node (h) at (0,1,1) {$\E_{2,3,4}$};
 
 \draw [-latex] (a) -- (e);
 \draw [-latex] (e) -- (h);
 \draw [-latex] (h) -- (g);
 
 \draw [-latex] (a) -- (b);
 \draw [-latex] (b) -- (c);
 \draw [-latex,red] (c) -- (g);
 
 \draw [-latex,green] (a) -- (d);
 \draw [-latex] (d) -- (h);
 \draw [-latex,blue] (d) -- (c);
 
 \draw [-latex] (b) -- (f);
 \draw [-latex] (e) -- (f);
 \draw [-latex] (f) -- (g);
\end{tikzpicture}
\end{center}
\end{example}

Furthermore, the following diagram describes the quiver of the exact category \ref{quiver-def}, associated to each exact structure in the previous example.

Compare  also the biquivers in \cite{Ku}, example 4.56.
\color{black}
\begin{center}
\begin{figure}[h]
\centering
\begin{tikzpicture}[scale=0.5]
	\centering 
    \node[draw=none] at (2,-0.6){$Q(\A,\E_{2})$};
    \node[draw=none] at (15,-0.6){$Q(\A,\E_{1,2})$};
    \node[draw=none] at (11,6.4){$Q(\A,\E_{min})$};
    \node[draw=none] at (24,6.4){$Q(\A,\E_{1})$};
    \node[draw=none] at (2,17.6){$Q(\A,\E_{2,3,4})$};
    \node[draw=none] at (11,24.6){$Q(\A,\E_{3})$};
    \node[draw=none] at (15,17.6){$Q(\A,\E_{max})$};
    \node[draw=none] at (24,24.6){$Q(\A,\E_{1,3,5})$};
   \draw[-](0,0)--(0,4)--(4,4)--(4,0)--(0,0);
   \draw[-](0,13)--(0,17)--(4,17)--(4,13)--(0,13);
   \draw[-](13,0)--(13,4)--(17,4)--(17,0)--(13,0);
   \draw[-](13,13)--(13,17)--(17,17)--(17,13)--(13,13);
   \draw[-](9,7)--(9,11)--(13,11)--(13,7)--(9,7);
   \draw[-](22,7)--(22,11)--(26,11)--(26,7)--(22,7);
   \draw[-](9,20)--(9,24)--(13,24)--(13,20)--(9,20);
   \draw[-](22,20)--(22,24)--(26,24)--(26,20)--(22,20);
   \draw[->](5,15)--(12,15);
   \draw[->](2,5)--(2,12);
   \draw[->](15,5)--(15,12);
   \draw[->,green](11,12)--(11,19);
   \draw[->](24,12)--(24,19);
   \draw[<-](5,5)--(8,8);
   \draw[<-](18,5)--(21,8);
   \draw[->](14,9)--(21,9);
   \draw[->](5,2)--(12,2);
   \draw[<-](5,18)--(8,21);
   \draw[<-,red](18,18)--(21,21);
   \draw[->,blue](14,22)--(21,22);  
   \node[draw=none] at (9.5,9){$S_{2}$};
   \node[draw=none] at (10.5,10){$P_{1}$};
   \node[draw=none] at (10.5,8){$P_{3}$};
   \node[draw=none] at (11.5,9){$I_{2}$};
   \node[draw=none] at (12.5,10){$S_{3}$};
   \node[draw=none] at (12.5,8){$S_{1}$};
   \draw[->,dotted](9.75,9.25)--(10.25,9.75);
   \draw[->,dotted](9.75,8.75)--(10.25,8.25);
   \draw[->,dotted](10.75,9.75)--(11.25,9.25);
   \draw[->,dotted](10.75,8.25)--(11.25,8.75);
   \draw[->,dotted](11.75,9.25)--(12.25,9.75);
   \draw[->,dotted](11.75,8.75)--(12.25,8.25);
   \node[draw=none] at (9.5,22){$S_{2}$};
   \node[draw=none] at (10.5,23){$P_{1}$};
   \node[draw=none] at (10.5,21){$P_{3}$};
   \node[draw=none] at (11.5,22){$I_{2}$};
   \node[draw=none] at (12.5,23){$S_{3}$};
   \node[draw=none] at (12.5,21){$S_{1}$};
   \draw[->,dotted](9.75,22.25)--(10.25,22.75);
   \draw[->,dotted](9.75,21.75)--(10.25,21.25);
   \draw[<-,dotted](10.75,22.75)--(11.25,22.25);
   \draw[->,dotted](10.75,21.25)--(11.25,21.75);
   \draw[->,dotted](11.75,22.25)--(12.25,22.75);
   \draw[->,dotted](11.75,21.75)--(12.25,21.25);
   \draw[<-](9.85,22)--(11.15,22);
   
   \node[draw=none] at (22.5,9){$P_{3}$};
   \node[draw=none] at (23.5,10){$S_{3}$};
   \node[draw=none] at (23.5,8){$S_{2}$};
   \node[draw=none] at (24.5,9){$P_{1}$};
   \node[draw=none] at (25.5,8){$S_{1}$};
   \draw[->,dotted](22.75,9.25)--(23.25,9.75);
   \draw[<-,dotted](22.75,8.75)--(23.25,8.25);
   \draw[->](23.75,9.75)--(24.25,9.25);
   \draw[->,dotted](23.75,8.25)--(24.25,8.75);
   \draw[->,dotted](24.75,8.75)--(25.25,8.25);
   \node[draw=none] at (0.5,2){$P_{1}$};
   \node[draw=none] at (1.5,3){$S_{1}$};
   \node[draw=none] at (1.5,1){$S_{2}$};
   \node[draw=none] at (2.5,2){$P_{3}$};
   \node[draw=none] at (3.5,1){$S_{3}$};
   \draw[->,dotted](0.75,2.25)--(1.25,2.75);
   \draw[<-,dotted](0.75,1.75)--(1.25,1.25);
   \draw[->](1.75,2.75)--(2.25,2.25);
   \draw[->,dotted](1.75,1.25)--(2.25,1.75);
   \draw[->,dotted](2.75,1.75)--(3.25,1.25);
   \node[draw=none] at (22.5,22){$S_{1}$};
   \node[draw=none] at (24,22){$S_{2}$};
   \node[draw=none] at (25.5,22){$S_{3}$};
   \node[draw=none] at (23.25,21){$P_{1}$};
   \draw[<-,dotted,thick](22.4,21.6)--(22.85,21);
   \draw[<-,dotted,thick](23.65,21)--(24.1,21.6);
   \draw[<-](24.45,22)--(25.15,22);
   \draw[<-](23.95,21)--(25.15,21.9);
   \node[draw=none] at (13.5,3){$S_{1}$};
   \node[draw=none] at (15,0.75){$S_{2}$};
   \node[draw=none] at (16.5,3){$S_{3}$};
   \node[draw=none] at (14.25,2){$P_{1}$};
   \node[draw=none] at (15.75,2){$P_{3}$};
   \draw[<-,dotted](13.5,2.5)--(13.75,2);
   \draw[<-,dotted](14.25,1.75)--(14.65,0.75);
   \draw[->,dotted](16.25,2)--(16.5,2.5);
   \draw[->](13.75,3)--(15.35,2.25); 
   \draw[->](16.25,3)--(14.65,2.25); 
   \draw[->,dotted](15.35,0.75)--(15.75,1.75); 
   \node[draw=none] at (0.5,15){$S_{1}$};
   \node[draw=none] at (2,15){$S_{2}$};
   \node[draw=none] at (3.5,15){$S_{3}$};
   \node[draw=none] at (2.75,14){$P_{3}$};
   \draw[->,dotted](3.15,14)--(3.5,14.5);
   \draw[<-,dotted](2.45,14)--(2,14.5);
   \draw[->](0.85,15)--(1.65,15);
   \draw[->](0.85,15)--(2.05,14);
   \node[draw=none] at (13.5,15){$S_{1}$};
   \node[draw=none] at (15,15){$S_{2}$};
   \node[draw=none] at (16.5,15){$S_{3}$};
   \draw[->](13.85,15)--(14.65,15);
   \draw[<-](15.35,15)--(16.15,15);
  
\end{tikzpicture}
\end{figure}  

\end{center} 

We can see that the path of matrix reductions discussed in section \ref{subsection:matrix reduction}
corresponds to the chain of exact structures 
$$\E_{max} \; \supset \; \E_{1,3,5} \; \supset \; \E_{3} \; \supset \; \E_{min} .$$
In fact, the ad hoc notion of a quiver of a matrix problem, given by the algebra operating on the current reduced form, together with arrows of degree one corresponding to the unreduced blocks, can finally be made precise: the reduction of exact structures transforms the Gabriel quiver $Q(\A,\E_{max})$ into the Auslander-Reiten quiver $Q(\A,\E_{min })$, and, in the first reduction step, the quiver of the exact category $(\A,\E_{1,3,5})$ coincides with the quiver depicted after the first reduction step in \ref{subsection:matrix reduction}. 
We only need to make precise why the exact category $(\A,\E_{1,3,5})$ corresponds to reducing the block $A$ of the matrix problem $M=[A|B]$. This is done in the next section \ref{subsection:new exact structure}.

\subsection{Constructing new exact structures from given ones}\label{subsection:new exact structure}

One method to produce exact structures is using exact functors, see also \cite[section 1.4]{DRSS}:
\begin{definition}\label{reduced-structure}
Let $(\mathcal{A}, \E_{\A})$ and $(\mathcal{B},\E_{\B})$ be exact categories and let $F:(\mathcal{A}, \E_{\A})\to (\mathcal{B},\E_{\B})$ be an exact functor, that is, the image $(Fi,Fd)$ of each exact pair $(i,d)$ in $(\A,\E_\A)$ is exact in $(\B,\E_\B)$. 
We define the following subclass of $\E_\A$:
\[\E_{F}= \{ \xi \in \E_\A \;:\; 0\rightarrow A\rightarrow B \rightarrow C\rightarrow 0 \; |\; F(\xi) \mbox{ is  split exact in } \B \} \; \]
\end{definition}
The following is a reformulation in our context of \cite[Lemma 1.9]{DRSS}, and it also follows from  \cite[7.3]{He}, see \cite[Exercise 5.3]{Bu}.
\begin{proposition}\label{exact reduction}
Let $F:(\mathcal{A}, \E_{\A})\to (\mathcal{B},\E_{\B})$ be an exact functor. Then
$(\mathcal{A}, \E_{F})$ is an exact category.
\end{proposition}
\begin{proof}We verify that $\E_{F}$ satisfies the axioms of an exact structure on $\A$. Since $F(1_{A})= 1_{FA}$ for every object $A$ of $\mathcal{A}$ and the identity is admissible monic and epic, $\E_{F}$ satisfies (E0) and (E0)$^{op}$.

An admissible monic in $\E_{F}$  is a morphism $i$ in a pair $(i,d)$ in $ \E_\A$ such that $F(i)$ is an admissible monic for the split exact structure $\E_{min}(\B)$ on $\B$. Since $\E_{min}(\B)$ is closed under composition of admissible monics  we conclude that $\E_{F}$ satisfies (E1).  The dual argument applies to admissible epics.

Now let us verify that $\E_{F}$ satisfies (E2) and (E2)$^{op}$:
 The push-out of an admissible monic $i: A \mono B$ in  $\E_{F}$ along an arbitrary morphism $f: A \to X$  exists in the exact category $(\A,\E_\A)$: 
\[\xymatrix{
A \; \ar[d]_{f} 
\ar@{ >->}[r]^{i}  \ar@{}[dr]|{\text{PO}} 
& B\ar[d]^{g}\\
X \; \ar@{->}[r]^{j} & D}
\]
We need to verify that $j$ is an admissible monic not only for $\E_\A$, but also for $\E_F$.
Consider the commutative diagram in $(\A,\E_\A)$
\[
\xymatrix{
A \ar@{ ->}[r]^i
\ar@{->}[d]_{f}
& B \ar@{->}[r]^d\ar@{->}[d]^g & C\ar@{=}[d] && \xi \\
X\ar@{->}[r]^{j} & D\ar@{->}[r]^{d'} & C  && {\xi}' = f \cdot \xi }
\]
which is mapped under the functor $F$ to the commutative diagram in $\B$
\[
\xymatrix{
F(A) \ar@{ ->}[r]^{F(i)}
\ar@{->}[d]_{F(f)}
& F(B) \ar@{->}[r]^{F(d)}\ar@{->}[d]^{F(g)} & F(C)\ar@{=}[d] && F(\xi) \in \E_{min}(\B)\\
F(X)\ar@{->}[r]^{F(j)} & F(D)\ar@{->}[r]^{F(d')} & F(C)  && F(\xi') }
\]

Since $i$ is an admissible monic, we know that $F(i)$ is a section and $F(d)$ a retraction. 
By commutativity, $F(d')F(g) = F(d)$ is a retraction, so $F(d')$ is a retraction. 
Since $F$ is exact, the pair $(F(j),F(d'))$ is exact, hence $F(j)$ is a section and (E2) holds.

The dual argument applied to the pull-back diagram of an admissible epic yields  (E2)$^{op}$.
\end{proof}

The basic idea of matrix reduction is to fix a subproblem and completely reduce representations of this subproblem into direct sums of indecomposables. On the level of exact structures, having nothing but direct sums of indecomposables corresponds to the choice of the split exact structure. 
Thus, if the functor $F$ in definition \ref{reduced-structure} is the projection onto a suitable subcategory (like representations of a subquiver or modules over a subalgebra), the definition of the exact structure $\E_F$ corresponds to the idea that objects in the subcategory are completely reduced into sums of indecomposables (we consider those exact sequences $\xi$ whose projection  $F(\xi)$  is  split exact).
\medskip

In  \cite[section 4]{DRSS}, several classical reductions are discussed, like one-point extension of an algebra or reduction of modules to a vector space problem. The underlying procedure is always the same:  complete reduction on a subproblem corresponds to the choice of an exact structure on the original problem, composed of those short exact sequences that split when restricted to the subproblem. We refer to \cite{DRSS} for more details. 
However the examples discussed there consider only {\em one} choice of exact structure on the original category $\mod A$. 
We propose to iterate this process (as it is done for matrix reductions or for Roiter's bocses) that is, to consider a chain of exact structures on the same underlying category $\A$.
\bigskip

We return now to the example \ref{cube}, 
 the category $\A = \rep Q$ of representations of  the quiver
\[ Q : \qquad \xymatrix{ 1 \ar[r]^{\alpha}  & 2  & 3 \ar[l]_{\beta}}\]

Consider $\A$ equipped with the abelian exact structure $\E_{max}$, and set up the first reduction step: Let $\B= \rep Q'$ be the category of representations of the subquiver 
\[ Q' : \qquad \xymatrix{ 1 \ar[r]^{\alpha}  & 2  }\]
of $Q$, and let 
$$F : \rep Q \to \rep Q'$$
be the restriction functor. 
Thus the exact structure $\E_F$ on $\A$ is given by all
short exact sequences $0\rightarrow U\rightarrow V\rightarrow W\rightarrow 0$  in $\rep Q$ whose restriction to the subquiver $Q'$ is split. 
It is not difficult to verify that exactly the non-split short exact sequences numerated 1, 3 and 5 from the table in example \ref{cube} are those whose restriction to $Q'$ splits. We therefore  conclude that 
$$\E_F = \E_{1,3,5}.$$
The matrix reduction step (1) in section \ref{subsection:matrix reduction}
was to reduce the matrix $A$. In view of the theory developed by now, this corresponds to choosing the exact structure which splits on the subquiver supported by the arrow $\alpha$, that is, the quiver $Q'$. Therefore, the reduction step (1) corresponds precisely to the reduction of exact structures
$$ \E_{max} \to \E_{1,3,5}$$
on $\A =\rep Q$.

\section{The lattice of exact structures of an additive category}\label{lattice}

\begin{definition}Let $\A$ be an additive category. We denote by $(Ex({\A}), \subseteq)$ the poset of exact structures $\E$ on $\A$, where the partial order is given by containment $\E' \subseteq \E$.
\end{definition}
This \emph{containment} partial order is the \emph{reduction of exact structures} discussed in \ref{def reduction}.
\begin{lemma}\label{intersection} For a family of exact structures  $(\E_{\omega})_{\omega \in \Omega}$ on an additive category $\A$, the intersection
\[\bigcap_{\omega \in \Omega} \E_{\omega}  = \{ \xi| \xi \in \E_{\omega} \mbox{ for all } \omega \in \Omega\}\]
forms an exact structure on $\A$.
\end{lemma}
\begin{proof} Let us show that this class verifies the axioms of the definition \ref{Quillen def}: 
(E0), (E0)$^{op}$, (E1) and (E1)$^{op}$ are satisfied since every $\E_{\omega}$ satisfies these axioms. 
For (E2), the push-out of an admissible monic $i$ in $\E_{\omega} $ exists in $\E_{\omega}$ and yields an admissible monic $f_i$ in $\E_{\omega}$ for all $\omega \in \Omega$:
\[\xymatrix{
A \; \ar[d]_{j} 
\ar@{ >->}[r]^{i}  \ar@{}[dr]|{\text{PO}} 
& B\ar[d]^{g}\\
C \; \ar@{>->}[r]^{f_i} & D}
\]
Since the push-out is unique up to isomorphism, and an exact structure is closed under isomorphisms, we conclude that (E2)  satisfied. Dually for (E2)$^{op}$.
\end{proof}

\begin{theorem}The poset of exact structures of an additive category $\A$ is a lattice $(Ex({\A}), \subseteq, \bigwedge, \bigvee)$.
\end{theorem}
\begin{proof}
Using lemma \ref{intersection} we define the following two binary operations on the poset $Ex({\A})$;
\emph{the meet} $\bigwedge$ is defined by $ \E_{\omega}\bigwedge \E_{\omega'} = \E_{\omega} \cap \E_{\omega'} $, 
and \emph{the join} $\bigvee$ is defined by
$$ \E_{\omega}\bigvee \E_{\omega'} = \bigcap \{ \E \in Ex(\A) \;|\; \E_{\omega}\subseteq \E, \E_{\omega'}\subseteq \E\}.$$ 
Note that the intersection defining the join is not an empty set since it contains the maximal exact structure $\E_{max}$ of the additive category $\A$; see \ref{Emax} for the existence of $\E_{max}$.
We conclude that  the poset $Ex({\A})$ is a lattice since it is a $\bigwedge$-semilattice and a $\bigvee$-semilattice. 

\end{proof}

\begin{cor}The lattice $(Ex({\A}), \subseteq, \bigwedge, \bigvee)$ of exact structures of an additive category is bounded and complete.
\end{cor}
\begin{proof}The lattice is bounded since it has a \emph{top} $\E_{max}$ and a \emph{bottom} $\E_{min}$ verifying 
$$\E \bigwedge \E_{max}=\E \mbox{ and } \E \bigvee \E_{min}= \E$$ for any exact structure $\E$ in $Ex({\A})$. And it is complete since all subsets $\{(\E_{\omega})_{\omega\in\Omega}\}$ of $Ex({\A})$ have both a \emph{meet} $\bigwedge (\E_{\omega})_{\omega\in \Omega} =\cap (\E_{\omega})_{\omega\in \Omega}$
and a \emph{join} defined by $\bigvee(\E_{\omega})_{\omega\in \Omega}= \bigcap\{\E | \E_{\omega}\subseteq \E, \forall \omega\in \Omega\}$, by lemma \ref{intersection}.
\end{proof}

\begin{example}
As seen in example \ref{triangulated}, if $\A$ is a triangulated category, then the lattice of exact structures is a single point: $Ex(\A)=\{\E_{min}\}$.
\end{example}

\begin{example}\label{A3}
Consider the category $\A = \rep Q$ of representations of  the quiver
\[ Q : \qquad \xymatrix{ 1 \ar[r]^{\alpha}  & 2  & 3 \ar[l]_{\beta}}\]
then the lattice of exact structures $Ex(\A)$ is the \emph{cube}  we construct in the example \ref{cube}.
Let us mention that by taking other forms of the quiver of type $A_{3}$ such as
\[ Q : \qquad \xymatrix{ 1   & 2 \ar[l]_{\alpha} \ar[r]^{\beta} & 3 }\]
or 
\[ Q : \qquad \xymatrix{ 1 \ar[r]^{\alpha}  & 2 \ar[r]^{\beta} & 3 }\]
we get a similar \emph{cube} (that is, a Boolean lattice) for $Ex(\A)$.
\label{cubeA3}\\
 In general, if we have $n$ Auslander-Reiten sequences then we have exactly $2^n$ exact structures which is the power set cardinality of an $n$ elements set.
In fact, the lattice is Boolean for a large class of exact categories:
\end{example}

\color{black}
\begin{theorem}\label{Grenouille verte}\cite{Enomoto} Let $\A$ be a skeletally small, Hom-finite, idempotent complete additive category which has finitely many indecomposable objects up to isomorphism. Then the lattice of exact structures $Ex(\A)$ is Boolean.

In fact, the set of exact structures on $\A$ is in bijection with the power set of Auslander-Reiten sequences in $\A$.
\end{theorem}
\begin{proof} This follows directly from 2.7 (see also 3.1, 3.7 and 3.10) in the work of Enomoto \cite{Enomoto}.
\end{proof}

\section{Length function on the poset $Obj\A$}\label{section 6}


The aim of this section is to define and study the notion of length for objects of an exact category $(\A, \E)$. Contrary to abelian categories, the Jordan-H\"older property does not hold for general exact(additive) categories, which makes it impossible to define length using composition series.

Throughout this section, we assume that $\A$ is essentially small, and we denote by $Obj\A$ the set of ismorphism classes of objects in $\A$. 
We show that the notion of $\E-$subobjects allows to turn $Obj\A$ into a poset, and that the length of an object corresponds to the height function of this poset.
Since the exact structure $\E$ is closed under isomorphisms, we work mostly with objects $X$ rather than their isomorphism classes $[X] \in Obj\A$.

\subsection{The length function}

\begin{definition}\label{lE}
We define the $\mathcal{E}-$length function $l_{\mathcal{E}}: Obj\mathcal{A}\rightarrow \mN \cup \{ \infty\}$ as supremum over the lengths of  chains of admissible monics which are not isomorphisms. 

That is, for an object $X$ of $(\A,\E)$, one has $l_{\mathcal{E}}(X) = n \in \mN$ if $n$ is the  maximal length of a chain of admissible monics which are not isomorphisms 

\[ 0=X_0 \; \mono X_1 \;\mono  \cdots \; \mono \;  X_{n-1}\;\mono\; X_n=X .\] 
We say in this case that $X$ has finite $\E-$length, or that $X$ is $\E-$finite. 
If no such bound exists, we say that $X$ has infinite length, or $l_\E (X) = \infty.$
Clearly, isomorphic objects have the same length, and therefore this definition gives rise to a function $l_{\mathcal{E}}: Obj\mathcal{A}\rightarrow \mN \cup \{ \infty\}$ defined on isomorphism classes.
\end{definition}

\begin{remark} The $\mathcal{E}-$simple objects are  precisely those of length $l_{\E}(X)=1$.
\end{remark}

\begin{example}\label{length ex} We illustrate how the $\E-$length of an object changes with the exact structure by considering the indecomposable injective representation $I_2$ from the example discussed in \ref{cube}, and measure its length with respect to various exact structures from $Ex(\A)$, see \ref{cube}:
 \[l_{\mathcal{E}_{min}}(I_2)= 1 \]
 \[ l_{\mathcal{E}_{1,3,5}}(I_2)=2 \] 
 \[ l_{\mathcal{E}_{ab}}(I_2)=3.\]
\end{example}
\bigskip

We call an exact category $(\A, \E)$ {\em finite} if every object is $\E-$finite. 
This implies the condition that $\A$ is an $\E-$Artinian and $\E-$Noetherian category in the following sense:

\begin{definition}
An object $X$ of $(\A, \E)$ is $\E-$Noetherian if any increasing sequence of $\E-$subobjects of $X$

\[X_1 \;\mono X_2  \; \mono \cdots \mono \;  X_{n-1}\;\mono\; X_n\mono X_n \; \cdots \]
becomes stationary.
Dually, an object $X$ of $(\A, \E)$ is $\E-$Artinian if any descending sequence of $\E-$subobjects of $X$

\[\cdots \; X_n \mono X_n \;  \mono X_{n-1}\;\mono  \cdots \; \mono X_2 \mono \;  X_1\;\] 
becomes stationary.
The exact category $(\A, \E)$ is called $\E-$Artinian 
(respectively $\E-$Noetherian) if every object is $\E-$Artinian (respectively \mbox{$\E-$Noetherian)}.
\end{definition}

\begin{lemma} Let $(\A, \E)$ be an exact category. An $\E-$finite object $X$ of $(\A, \E)$ is $\E-$Artinian and $\E-$Noetherian.
\end{lemma}

\begin{proof}
For an $\E-$finite object $X$ of length $l_{\E}(X)=n \in \N$, the longest chain of proper admissible monics is of length $n$. Thus any increasing or decreasing sequence of $\E-$subobjects of $X$ must become stationary and $X$ is $\E-$Artinian and $\E-$Noetherian.
\end{proof}
For the opposite direction, we have the following result:
\begin{lemma}Let $(\A, \E)$ be an exact category. An $\E-$Artinian and $\E-$Noetherian object $X$ of $(\A, \E)$ admit an $\E-$composition series.

\end{lemma}
\begin{proof}
The detailed proof appeared in \cite[Lemma 7.5]{BHT}.
\end{proof}
We now study how the length function behaves with respect to short exact sequences: It turns out to be a superadditive function. We provide in \ref{nonJH} an example that it need not be additive in general.
\begin{theorem}\label{lXYZ}
Let $ X \imono{f} Y \iepi{g} Z $ be a short exact sequence of $\E-$finite objects. Then \[l_{\E}(Y) \ge l_{\E}(X)+l_{\E}(Z).\]
\end{theorem}

\begin{proof} Consider a chain of proper admissible monics which defines the length $s$ of $Z$:
\[ 0=Z_0 \; \imono{i_1} Z_1 \;\mono  \cdots \; \mono \;  Z_{s-1}\;\imono{i_s}\; Z_s=Z. \]
Denote by $Y_{s-1}$ the pull-back of $g$ along $i_s$. 
By the the dual of \cite[Prop 2.12]{Bu}, there exists a commutative diagram with exact columns
$$\xymatrix{
X  \ar@{ >->}[d]_{f_{s-1}}  \ar@{=}[r]  & X \ar@{ >->}[d]^{f}\\
Y_{s-1} \; \ar@{ ->>}[d]_{g_{s-1}}  \ar@{ >->}[r]^{j_{s}}  & Y\ar@{ ->>}[d]^{g}\\
Z_{s-1} \; \ar@{>->}[r]_{i_s} & Z}$$
Since $i_s$ is an admissible monic, \cite[Prop 2.15]{Bu} yields that $j_s$ is one as well, and since $i_s$ is not an isomorphism, $j_s$ cannot be an isomorphism by \cite[3.3]{Bu}.
Iterated pull-backs along the morphisms $g_{s-1}, g_{s-2}, \ldots, g_1$ therefore yield the following exact diagram with exact columns and proper admissible monics $j_1, \ldots j_{s}$:
$$\xymatrix{
X  \ar@{ >->}[d]_{f_{0}}  \ar@{=}[r]  & X \ar@{ >->}[d]_{f_{1}} \ar@{.}[r] & X  \ar@{ >->}[d]_{f_{s-1}}  \ar@{=}[r]  & X \ar@{ >->}[d]^{f}\\
Y_{0} \; \ar@{ ->>}[d]_{g_{0}}  \ar@{ >->}[r]^{j_{1}}  &Y_{1} \ar@{.}[r] \ar@{ ->>}[d]_{g_{1}} & Y_{s-1} \; \ar@{ ->>}[d]_{g_{s-1}}  \ar@{ >->}[r]^{j_{s}}  & Y\ar@{ ->>}[d]^{g}\\
0=Z_0\;\ar@{>->}[r]_{i_1} & Z_1 \ar@{.}[r] & Z_{s-1} \; \ar@{>->}[r]_{i_s} & Z}$$

Note that $f_0$ is admissible monic with zero cokernel, hence an isomorphism by remark \ref{zero coker}.
Composing the sequence of proper admissible monics $j_1, \ldots j_{s}$  with any sequence of proper subobjects of $X$: \[ 0=X_0 \; \mono X_1 \;\mono  \cdots \; \mono \;  X_{r-1}\;\mono\; X_r=X \]
yields a chain of proper admissible monics ending in $Y$ of length $r+s$. Hence, the definition of length yields $l_{\E}(Y) \ge  l_{\E}(X)+l_{\E}(Z)$.
\end{proof}

\begin{cor} Let $Y\iepi{g} Z$ be an admissible epic between $\E-$finite objects. Then $l_{\E}(Y)\geq l_{\E}(Z)$.
\end{cor}
\begin{proof}
The kernel of $g$ yields the  short exact sequence $\ker g \mono Y \iepi{g} Z$. Hence theorem \ref{lXYZ} implies that $l_{\E}(Y) \ge l_{\E}(Z)$ since $l_{\E}(\ker g)\geq 0$.
\end{proof}

\begin{remark}
Analogously to abelian categories, one could define a {\em composition series} of an object $X$ to be a chain of admissible monics 

\[ 0=X_0 \; \imono{i_1} X_1 \;\imono{i_2}  \cdots \; \imono{i_{n-1}} \;  X_{n-1}\;\imono{i_{n}} \; X_n=X \] 
whose cokernels are $\E-$simple. These composition series are certainly chains of proper admissible monics that cannot be refined, so they are good candidates for chains defining the length of $X$.
However, the length of a composition series of an object $X$ need not be unique in general, that is, the Jordan-H\"older property does not hold necessarily. We provide a simple example:
\end{remark}

\begin{example}\label{nonJH}
Consider the split exact structure $\E=\E_{min}$. As seen in remark \ref{E-simple indecomposable}, the $\E-$simple objects are precisely the indecomposables. Hence in this case the $\E-$length function measures the maximum number of indecomposable direct summands of an object $X$. 
The Jordan-H\"older property thus coincides with the Krull-Schmidt property, and we obtain a counterexample re-visiting example \ref{submonoid2}:
The category $\A_S$  for $S = \N \backslash \{ 1 \}$ equipped with the split exact structure admits two $\E-$simple objects, $k^2$ and $k^3$, up to isomorphism.
There are two composition series for the object $X =k^6$ in $\A_S$, one of length 3 with cokernels $k^2$, the other of length 2 with cokernels $k^3$.
Following our definition, the object $X =k^6$ has length 
$l_{\E}(X) = 3$.

This example also shows that the length function need not be additive on short exact sequences: Consider the short exact sequence
$$0\rightarrow k^3 \mono k^6 \epi k^3 \rightarrow 0$$ 
in $(\A_S,\E)$, then
\[ l_{\E}(k^6) = 3 \neq 2 = l_{\E}(k^3)+l_{\E}(k^3) .\]
\end{example}

\subsection{The poset structure on $Obj\A$}

We assume in this section that $(\A, \E)$ is a finite  exact category, that is, every object is $\E-$finite.
In general the length function behaves well with respect to $\E-$subobjects, that is $l_\E(X) \le l_\E(Y)$ if $X \subset_\E Y$. The following lemma shows that strict inclusion is also
preserved when the objects are of \emph{finite} length:
\begin{lemma}\label{lengthXY}

Consider two objects $X$ and $Y$ in $A$ such that $X{\subsetneq}_{\mathcal{E}} 
Y$. Then
\[ l_{\mathcal{E}}(X)\lneqq l_{\mathcal{E}}(Y).\] 
\end{lemma}
\begin{proof}
Let $X$ be a proper admissible subobject of $Y$,  that is, there exists an admissible monic $X\imono{i}Y$ which is not an isomorphism. We show that $$l_{\mathcal{E}}(X)+ 1 \leq l_{\mathcal{E}}(Y).$$
Assume that $X$ has length $l_{\mathcal{E}}(X)=n$. Extending a chain of subobjects defining $l_{\mathcal{E}}(X)$, we obtain a sequence of proper admissible monics ending via $i$ in $Y$ of the following form: \[ 0=X_0 \; \mono X_1 \;\mono  \cdots \; \mono \;  X_{n-1}\;\mono X=X_n\imono{i} Y. \] 
Thus the length of $Y$ is at least $n+1$.
\end{proof}

The previous lemma allows us to show that the notion of $\E-$subobjects turns $Obj\A$ into a poset:
\begin{proposition}\label{poset}
Let $(\A, \E)$ be a finite essentially small exact category. Then
the relation $\subset_\E $ induces a partial order on $Obj\A$.
\end{proposition}
\begin{proof}
We defined the relation $\subset_\E $ on objects, but since the exact structure $\E$ is closed under isomorphisms, one also obtains a well-defined relation on the set of isomorphism classes $Obj\A$. It remains to show that this relation verifies the three properties defining a partial order. We do so mostly by working with objects $X$ rather than their isomorphism classes $[X]$.
\begin{enumerate}

\item Reflexive: $X\subset_\E X$ since the identity $X{\imono{id_X}} X$ is an admissible monic by (E0).

\item Antisymmetric: Assume that
 $X\subset_\E Y$ and $Y\subset_\E X$. Then we have $l_{\mathcal{E}}(X)\leq l_{\mathcal{E}}(Y)$ and $l_{\mathcal{E}}(Y)\leq l_{\mathcal{E}}(X)$, and so $l_{\mathcal{E}}(X)= l_{\mathcal{E}}(Y)$. Hence the admissible monic 
 $X\mono Y$ establishing $X \subset_\E Y$ cannot be proper by lemma \ref{lengthXY}, which shows $X=Y$ in $Obj\A$.
 
\item Transitive: if $X\subset_\E Y$ and $Y\subset_\E Z$ then there exist admissible monics $X\imono{f} Y$ and $ Y\imono{f'} Z$.  By (E1), $X\imono{f'\circ f}Z$ is an admissible monic and so $X\subset_\E Z.$
\end{enumerate}
\end{proof}

\begin{remark}
Now since we know that the notion of $\E-$subobject induces a poset structure on $Obj\A$, we could define the length of an object $X$ as the height of the element $[X]$ in $Obj\A$. In fact, this is exactly how we defined length (as maximum length of a chain of $\E-$subojects), except that we rather start with the zero object having length zero, instead of height one.
\end{remark}

We recall the following definition from \cite{Kr11}:
\begin{definition}\label{Ameasure}A {\em measure for a poset $\cS$} is a morphism of posets $\mu :\cS\rightarrow \mathcal{P}$ where $(\mathcal{P}, \leq)$ is a totally ordered set.
\end{definition}

\begin{theorem}The length function $l_{\mathcal{E}}$ of a finite essentially small exact category $(\A, \E)$ is a measure for the poset $Obj\mathcal{A}$.
\end{theorem}
\begin{proof}The length function $l_{\mathcal{E}}: Obj\mathcal{A}\rightarrow \mN $ is defined on the set $Obj\A$, which is a partially ordered set by proposition \ref{poset}. Moreover, $l_{\E}$ is a morphism of partially ordered sets, and so $l_{\mathcal{E}}$ is a measure since $(\mN, \leq)$ is totally ordered.
\end{proof}

\section{Gabriel-Roiter measure}

In his proof of the first Brauer-Thrall conjecture \cite{Roi}, Roiter used an induction scheme which Gabriel formalized in his report on abelian length categories \cite{Gab}. This so-called Gabriel-Roiter measure on module categories was further studied by Ringel in \cite{Ri05} and \cite{Ri06} in the representation-infinite case. Later Krause presented an axiomatic characterization of the Gabriel-Roiter measure on abelian length categories which reveals its combinatorial nature in \cite{Kr07} and \cite{Kr11}.
Our aim in this section is to extend the work of \cite{Kr11} to the more general context of exact categories.
Most of the results presented here generalize the corresponding version of Ringel or Krause.

In this section we consider $(\A,\E, ind \,\mathcal{A}, l_{\E})$ where $\A$ is an essentially small additive category, $\E$ is a fixed exact structure such that $(\A, \E)$ is a finite exact category,  $ind \,\mathcal{A}$ is the set of isomorphism classes of indecomposable objects of $\A$, and $l_{\E}$ is the associated length. 
The set $ind\, \A$ does not depend on the exact structure $\E$, but the partial order does depend on $\E$. We therefore write  $(ind\A,{\subset}_{\E})$ when referring to the poset. 
\subsection{The definition and existence}

The following definition extends the one from \cite[Definition 1.6]{Kr11} to the realm of exact categories: a Gabriel-Roiter measure on $(ind\A,{\subset}_{\E})$ is a morphism of partially ordered sets which refines the length function $l_{\E}$ and satisfies that the measure of an object $X$ cannot exceed the measure of an object $Y$ of at most equal length if all subobjects of X have smaller measure than $Y$:
\begin{definition}A map ${\mu}_{\E}:(ind\A,{\subset}_{\E}) \rightarrow (\mathcal{P},\leq)$ is called a Gabriel-Roiter measure on the exact category $(\A, \E)$ if it verifies the following axioms
\begin{enumerate}
\item[$(GR_1)$]$\mu_{\E}$ is a measure
\item[$(GR_2)$]$\mu_{\E}(X)=\mu_{\E}(Y)$ implies $l_{\E}(X)= l_{\E}(Y)$ for all $X, Y \in ind\,\A$
\item[$(GR_3)$]If $l_{\E}(X)\ge l_{\E}(Y)$ and $\mu_{\E}(X') \lneqq \mu_{\E}(Y)$ for all $X'{\subsetneq}_{\mathcal{E}}X$, then $$\mu_{\E}(X) \leq \mu_{\E}(Y).$$

\end{enumerate}
\end{definition}

Most constructions of a Gabriel-Roiter measure use as totally ordered set $(\mathcal{P},\leq)$ the set $\mathfrak{S}(\mathbb{N})$ of all vectors of natural numbers of finite length equipped with the lexicographic order $\lll$ on vectors with the natural order on $\mN$ reversed.
More explicitly, let $x=(x_1,...,x_n)$ and $y=(y_1,...,y_m)$ be two vectors of natural numbers.
We write $x \lll \nshortmid y$ if the element $x$ in the ordered set $\mathfrak{S}(\mathbb{N})$ is smaller but not equal to $y$. 
To compare these two vectors by $\lll$, we begin with the first elements; if for example $x_1=y_1$ we pass to the second elements, if again $x_2=y_2$ we pass to the third, and we continue like this until we obtain one of the following three cases:

\begin{enumerate}
\item if $x_k=y_k$ for all $1\leq k\leq i-1$ and at  position $i$ there are two different elements $x_i\lneqq y_i$ in $(\mathbb{N},\leq)$, then we get the inverse relation for the vectors: $(x_1,...,x_n) \nshortmid \ggg(y_1,...,y_m)$
\item if $n\lneqq m$ and $x_k=y_k$ for all $1\leq k\leq n$, then $(x_1,...,x_n) \lll \nshortmid (y_1,...,y_m)$
\item if $m=n$ and $x_k=y_k$ for all $1\leq k\leq m$, then $(x_1,...,x_n) = (y_1,...,y_m)$.
\end{enumerate}
\bigskip
More loosely speaking, one has $x\lll y$ if $x$ is a subword of $y$ in the sense of point $2$ above, or $y$ is denser than $x$ at the beginning, for example 
\[(1) \lll (1,3,4) \lll (1,2,4).\]
Let us now consider the following construction (we show later that it yields a Gabriel-Roiter measure for exact categories). For a fixed indecomposable object $X\in \A$, we consider the proper $\E-$filtrations $F_{\E}(X)$ of $X$ 
\[F_{\E}(X)= X_1{\subsetneq}_{\mathcal{E}}...{\subsetneq}_{\mathcal{E}} X_n=X\]
where all objects $X_i$ are indecomposable. Denote the vector of lengths in this filtration by 
\[ l_{\E}(F_{\E}(X))= (l_{\E}(X_1),..., l_{\E}(X)).\]

\begin{definition}\label{construction} Define a map
\[\mu_{\E}:(ind\A,{\subset}_{\E})  \rightarrow (\mathfrak{S}(\mathbb{N}), \lll)\]
by
$$X\longmapsto \mu_\E(X)={\max\limits_{F_{\E}(X)} \;
(l_{\E}(F_{\E}(X)))}$$
where the maximum is over all proper $\E-$filtrations of $X$ by indecomposables.
Note that the maximum is attained: We know by lemma \ref{lengthXY} that $l_{\E}(F_{\E}(X))$ is a strictly increasing sequence. But there are only finitely many strictly increasing sequences ending in the natural number $l_{\E}(X)$.
\end{definition}

\begin{example}\label{measure example}
Consider the split exact structure $\E_{min}$. Then all the indecomposable objects are $\E_{min}-$simples, and $l_{\E_{min}}(X)= 1$, therefore
\[{\mu}_{\E_{min}}(X)=(1) \]
for all $X\in ind\A$.

This is the case for example \ref{submonoid2}; $X=K^2$ or $X=K^3$ and then
\[{\mu}_{\E_{min}}(K^2)={\mu}_{\E_{min}}(K^3)=(1).\]
\end{example}

\bigskip

The following lemma can be derived from \cite[section 1]{Kr11}, applied to the length function $l_\E$ on the poset $(ind\A,{\subset}_{\E})$. We give a short proof in our setup.
\begin{lemma}$\label{measure}\mu_{\E}: (ind\A,{\subset}_{\E})  \rightarrow (\mathfrak{S}(\mathbb{N}), \lll)$ is a measure for $(ind\A,{\subset}_{\E})$.
\end{lemma}
\begin{proof}We have that $(ind\A, {\subset}_{\mathcal{E}})$ is a partially ordered set by the induced order on $ind\A \subset Obj\A$, and it is easy to see that $ (\mathfrak{S}(\mathbb{N}), \lll)$ is a totally ordered set since $(\mathbb{N},\leq)$ is totally ordered. It suffices to show that $\mu_{\E}$ is a morphism of posets. To this end, let $X {\subsetneq}_{\mathcal{E}} Y$, and consider a filtration \[F_{\E}(X): X_1{\subsetneq}_{\mathcal{E}}...{\subsetneq}_{\mathcal{E}}X_{n}=X\] such that \[{\mu}_{\E}(X)=l_{\E}(F_{\E}(X))= (l_{\E}(X_1),...,l_{\E}(X_n)).\]
This yields the following filtration of $Y$ \[F_{\E}(Y): X_1{\subsetneq}_{\mathcal{E}}...{\subsetneq}_{\mathcal{E}}X_{n}=X{\subsetneq}_{\mathcal{E}} Y\] with \[l_{\E}(F_{\E}(Y))= (l_{\E}(X_1),...,l_{\E}(X_n),l_{\E}(Y))\]
hence \[ {\mu}_{\E}(X)=(l_{\E}(X_1),...,l_{\E}(X_n))\lll (l_{\E}(X_1),...,l_{\E}(X_n),l_{\E}(Y))\lll {\mu}_{\E}(Y).\]
\end{proof}

The previous lemma establishes that the measure of a subobject $X'$ of $X$ is smaller than the measure of $X$. Of particular importance will be subobjects of $X$ whose measure is a subword of the measure of $X$, we call them as follows:
\begin{definition}A chain 
$$F_{\E}(X) : \; X_1 \,{\subsetneq}_{\mathcal{E}} \, X_2 \, {\subsetneq}_{\mathcal{E}}\, ... \,{\subsetneq}_{\mathcal{E}}X_{n-1} \,{\subsetneq}_{\mathcal{E}} \, X_n=X$$
in $ind\A$ is called a \emph{$\mu_\E-$filtration} of $X$ if for all $1\leq i \leq n$ the vector $\mu_\E(X_i)$ coincides with the subword of $\mu_\E(X)$ formed by the first $i$ entries.
\end{definition}
\begin{lemma}\label{submeasure}
Let $F_{\E}(X) : \; X_1 \,{\subsetneq}_{\mathcal{E}} \, X_2 \, {\subsetneq}_{\mathcal{E}}\, ... \,{\subsetneq}_{\mathcal{E}}X_{n-1} \,{\subsetneq}_{\mathcal{E}} \, X_n=X$ be a filtration of $X$ realizing the measure of $X$, that is,   ${\mu}_{\E}(X)=l_{\E}(F_{\E}(X))$. Then $F_{\E}(X)$ is a $\mu_\E-$filtration of $X$.

\end{lemma}
\begin{proof}
We have to show for each $1\leq i \leq n$ that ${\mu}_{\E}(X_i) = (l_{\E}(X_1),...,l_{\E}(X_i)).$
Of course, the sequence $(l_{\E}(X_1),...,l_{\E}(X_i))$ is one candidate for the maximum ${\mu}_{\E}(X_i)$, so we only need to show that the case $${\mu}_{\E}(X_i)=(l_{\E}(X'_1),...,l_{\E}(X'_m)) \nshortmid\ggg (l_{\E}(X_1),...,l_{\E}(X_i)) \mbox{ with } X'_m=X_i$$ 
is impossible. By definition of the order relation $\lll$, there are two situations to be considered:
\begin{enumerate}
\item there exists an index $1\leq j\leq \min\{i, m\}$ such that 
$$l_{\E}(X_k)=l_{\E}(X'_k) \mbox { for all } 1\leq k < j \mbox{ and } l_{\E}(X'_j) < l_{\E}(X_j).$$ 
But then the filtration of $X$
\[ X'_1{\subsetneq}_{\mathcal{E}}...{\subsetneq}_{\mathcal{E}}X'_m=X_i {\subsetneq}_{\mathcal{E}} X_{i+1}{\subsetneq}_{\mathcal{E}} ...{\subsetneq}_{\mathcal{E}} X_{n}=X\]\label{filtX}
yields a length sequence which is denser in the beginning than $\mu_{\E}(X)$, which  contradicts the fact that $F_{\E}(X)$ realizes the measure of $X$.

\item The sequence $(l_{\E}(X_1),...,l_{\E}(X_i))$ is a subword of $(l_{\E}(X'_1),...,l_{\E}(X'_m))$, that is,  $i\lvertneqq m$ and $l_{\E}(X_k)=l_{\E}(X'_k)$ for all $1\leq k\leq i$. 
But then again the same filtration of $X$ in \ref{filtX} yields a contradiction.
\end{enumerate}
\end{proof}
\begin{theorem}\label{GR measure}
(Compare \cite[Section 3.1]{Kr11}).
The map $\mu_{\E}$ is a Gabriel-Roiter measure  for $ind(\A,\E)$.
\end{theorem}

\begin{proof}
We verify that $\mu_{\E}$ as given in definition \ref{construction} satisfies the three axioms of a Gabriel-Roiter measure.

$(GR_1):$ This is lemma \ref{measure}. 

$(GR_2):$ If  
$ {\mu}_{\E}(X)=(l_{\E}(X_1),...,l_{\E}(X_n))=(l_{\E}(Y_1),...,l_{\E}(Y_m))={\mu}_{\E}(Y)$ then clearly 
$l_{\E}(X)=l_{\E}(X_n)=l_{\E}(Y_m)=l_{\E}(Y)$.

$(GR_3):$  Let $X$ and $Y$ be such that $l_{\E}(X)\geq l_{\E}(Y)$ and ${\mu}_{\E}(X')\lll \nshortmid {\mu}_{\E}(Y)$ for all $X'{\subsetneq}_{\mathcal{E}} X$.
Let ${\mu}_{\E}(Y)=(l_{\E}(Y_1),...,l_{\E}(Y_m))$ and ${\mu}_{\E}(X)=(l_{\E}(X_1),...,l_{\E}(X_n))$. Assuming that ${\mu}_{\E}(X) \nshortmid \ggg {\mu}_{\E}(Y)$,  we have one of the following cases:
\begin{enumerate}

\item there exists $1\leq i\leq m$ such that $l_{\E}(Y_k)= l_{\E}(X_k)$ for all $1\leq k \leq i-1$ and $l_{\E}(Y_i)\gneqq l_{\E}(X_i)$. But we know from lemma \ref{submeasure} that 
$${\mu}_{\E}(X_i)=(l_{\E}(X_1),...,l_{\E}(X_i)),$$ thus $ {\mu}_{\E}(Y)\lll \nshortmid  {\mu}_{\E}(X_i) $ and we get a contradiction by taking $X'=X_i.$

\item $m\lneqq n$ and ${\mu}_{\E}(Y)$ is a subword of ${\mu}_{\E}(X)$. Again by lemma  \ref{submeasure} we get 
\[{\mu}_{\E}(Y)= {\mu}_{\E}(X_m)=(l_{\E}(X_1),...,l_{\E}(X_m))\]
which yields a contradiction by choosing $X'=X_m$.
\end{enumerate}
Hence ${\mu}_{\E}(X) \lll {\mu}_{\E}(Y)$ and $(GR_3)$ is satisfied.
\end{proof}

\subsection{Some basic properties}
Krause shows in \cite{Kr11} that the Gabriel-Roiter measure satisfies some properties on abelian length categories, and we are studying here if these properties still hold for finite exact categories. Let $\mu_{\E}$ be the Gabriel-Roiter measure as in definition \ref{construction} for the finite exact category $(\A, \E)$.

\begin{proposition}$\mu_{\E}$ satisfies the following properties:
\begin{enumerate}
\item[$(GR_4)$] $\mu_{\E}(X)\lll \mu_{\E}(Y)$ or $\mu_{\E}(Y)\lll \mu_{\E}(X)$ for all $X, Y\in ind\A$.
\item[$(GR_5)$]$\{ \mu_{\E}(X)| X \in ind\A, l_{\E}(X)\leq n \}$ is a finite set for all $n\in \mathbb{N}$.
\item[$(GR_6)$]$X\in ind\A$ is $\E-$simple if and only if $\mu_{\E}(X)\lll \mu_{\E}(Y)$ for all $ Y\in ind\A$.
\end{enumerate}
\end{proposition}
\begin{proof}$(GR_4)$ is clear since $(\mathfrak{S}(\mathbb{N}), \lll)$ is totally ordered.
$(GR_5)$ follows from the fact that the set of strictly increasing vectors 
$$\{v \in \mathfrak{S}(\mathbb{N})\,|\, v=\mu_{\E}(X)=(v_1,...,l_{\E}(X))\}$$ 
is finite since $l_{\E}(X)\leq n$.
To prove $(GR_6)$ we need to remenber that $(\A, \E)$ is a finite exact category, so all objects are of finite length.  Hence each indecomposable object is $\E-$Artinian and thus has an $\E-$simple $\E-$subobject. Let us also note that each  indecomposable $\E-$simple object $X$ satisfies $\mu_{\E}(X)=(1)$.
\end{proof}

In the aim to show more properties of Gabriel-Roiter measure on finite exact categories, we extend the following definitions from \cite{Kr11} (3.3) for exact categories:
 \begin{definition}Let $X, Y\in ind\A$. We say that $X$ is a \emph{$\E-$Gabriel-Roiter predecessor} of $Y$ if $X {\subsetneq}_{\mathcal{E}} Y$ and $\mu_{\E}(X)= max_{Y'{\subsetneq}_{\mathcal{E}}Y} \mu_{\E}(Y')$.
 An inclusion $X{\subsetneq}_{\mathcal{E}}Y$ is called \emph{$\E-$Gabriel-Roiter inclusion} if X is a \emph{$\E-$Gabriel-Roiter predecessor} of $Y$. We denote it $X{\subset}_{\E}^{GR}Y$.
 \end{definition}
 Note that each object $Y\in ind_{\A}$ which is not $\E-$simple admits an $\E-$Gabriel-Roiter predecessor, by $(GR_4)$ and $(GR_5)$. An $\E-$Gabriel-Roiter predecessor $X$ of $Y$ is usually not unique, but the value $\mu_{\E}(X)$ is unique and 
 determined by $\mu_{\E}(Y)$.
\begin{definition}A chain 
\[X_1{\subsetneq}_{\mathcal{E}}X_2{\subsetneq}_{\mathcal{E}}...{\subsetneq}_{\mathcal{E}}X_{n-1}{\subsetneq}_{\mathcal{E}}X_n=X\]
in $ind\A$ is called a \emph{$\E-$Gabriel-Roiter filtration} of $X$ if $X_1$ is $\E-$simple and $X_{i-1}$ is an \emph{$\E-$Gabriel-Roiter predecessor} of $X_i$ for all $ 2\leq i \leq n$.
\end{definition}

\begin{proposition}
A chain 
\[X_1{\subsetneq}_{\mathcal{E}}X_2{\subsetneq}_{\mathcal{E}}...{\subsetneq}_{\mathcal{E}}X_{n-1}{\subsetneq}_{\mathcal{E}}X_n=X\]
in $ind\A$ is an $\E-$Gabriel-Roiter filtration of $X$ if and only if it is a $\mu_\E-$filtration of $X$.
\end{proposition}
\begin{proof}
Let F be a \emph{$\mu_\E-$filtration} of $X$ in $ind\A$.
\[F:= X_1{\subsetneq}_{\mathcal{E}}X_2{\subsetneq}_{\mathcal{E}} \cdots {\subsetneq}_{\mathcal{E}}X_{n-1}{\subsetneq}_{\mathcal{E}}X_n=X\]

Suppose F is not a Gabriel-Roiter filtration. Then for some i$\in$ \{1,...,n-1\}, $X_{i}$ is not a $\E-$Gabriel-Roiter predecessor of $X_{i+1}$, that is, there exists a subobject $X'$ of $X_{i+1}$ such that  $\mu_{\E}(X_{i})\lll \nshortmid \mu_{\E}(X')$. Let $F'$ and $F_{i}$ be filtrations \[F_{i}:= X_1{\subsetneq}_{\mathcal{E}}X_2{\subsetneq}_{\mathcal{E}}...{\subsetneq}_{\mathcal{E}}X_{i-1}{\subsetneq}_{\mathcal{E}}X_i\]

\[F':= X_1'{\subsetneq}_{\mathcal{E}}X_2'{\subsetneq}_{\mathcal{E}}...{\subsetneq}_{\mathcal{E}}X_{m-1}'{\subsetneq}_{\mathcal{E}}X_{m}'=X'\]
giving $\mu_{\E}(X_{i})$ and $\mu_{\E}(X')$ respectively.

Since both are subobject of $X_{i+1}$, we can complete both vectors of measure with $l(X_{i+1})$. In this situation, if $\mu_{\E}(X_{i})$ is a strict subword of $\mu_{\E}(X')$, then $X_{i+1}'$ being subobject of $X_{i+1}$ gives 
$(\mu_{\E}(X_{i}),l(X_{i+1}))\lll \nshortmid (\mu_{\E}(X'),l(X_{i+1}))$. 

On the other hand, if $X_{j}$=$X_{j}'$ for all j$\in$\{1,...,$l$-1\} and $l(X_{l}')\leq l(X_{l})$, then the completion of both vector is trivially order preserving. Both cases lead to a contradiction of $F$ being a \emph{$\mu_\E-$filtration}, thus a \emph{$\mu_\E-$filtration} is a Gabriel-Roiter filtration.
\medskip

Conversely, let us show that all Gabriel-Roiter filtrations are $\mu_\E-$filtrations. We proceed by induction on $m$. Of course the Gabriel-Roiter filtrations of length 1
coincide with the $\mu_\E-$filtrations of same length. Suppose now that the statement is true for all $l \in $ \{1,...,$m$-1\}.
Let $G$ and $F$:\[F:= X_1{\subsetneq}_{\mathcal{E}}X_2{\subsetneq}_{\mathcal{E}}...{\subsetneq}_{\mathcal{E}}X_{n-1}{\subsetneq}_{\mathcal{E}}X_n=X\]
\[G:= Y_1{\subsetneq}_{\mathcal{E}}Y_2{\subsetneq}_{\mathcal{E}}...{\subsetneq}_{\mathcal{E}}Y_{m-1}{\subsetneq}_{\mathcal{E}}Y_m=X\]
be two filtrations of X such that $F$ is a $\mu_\E-$filtration and $G$ is a Gabriel-Roiter filtration. We know that $Y_{m-1}$ is a Gabriel-Roiter predecessor of $X$, so $\mu_{\E}(X_{n-1}) \lll \mu_{\E}(Y_{m-1})$. By induction hypothesis, $\mu_{\E}(Y_{m-1})=(l(Y_{1}),l(Y_{2}),..., l(Y_{m-1}))$ since it is  given by a Gabriel-Roiter filtration of $Y_{m-1}$ of length $m-1$. By completing the vector with $l(X)$, using the same reasoning as above, we obtain that \[(l(X_{1}),l(X_{2}),..., l(X_{n-1}),l(X)) \lll (l(Y_{1}),l(Y_{2}),..., l(Y_{m-1}),l(X)).\] 
Since $F$ is a $\mu_\E-$filtration, we automatically get \[(l(X_{1}),l(X_{2}),..., l(X_{n-1}),l(X)) = (l(Y_{1}),l(Y_{2}),..., l(Y_{m-1}),l(X))\] and thus every Gabriel-Roiter filtration is a $\mu_\E-$filtration.

\end{proof}
\begin{proposition}$(GR_7)$ Suppose that $\mu_{\E}(X)\lll \nshortmid 
\mu_{\E}(Y)$. Then there are 

$Y'{\subsetneq}_{\mathcal{E}}Y"{\subset}_{\E}Y$ in $ind\A$ such that $Y'$ is a \emph{$\E-$Gabriel-Roiter predecessor} of $Y"$ with 
$\mu_{\E}(Y')\lll \mu_{\E}(X)\lll \nshortmid \mu_{\E}(Y")$ and $l_{\E}(Y')\leq l_{\E}(X)$.
\end{proposition}
\begin{proof}
The proof of $(GR_7)$ in \cite{Kr11} on abelian length categories can be generalized for finite exact categories, we adapt it by replacing each monomorphism by ${\subset}_{\E}$, and the length function by our length \ref{lE}. 
\end{proof}


Now we are studying, always in the more general context of essentially small exact categories, the main property of the Gabriel-Roiter measure due to \emph{Gabriel}, that is shown in \cite[3.4,$(GR_8)$]{Kr11} for abelian length categories. In fact we will see that  it does not always hold for all exact categories.

\begin{definition} Let $(\A, \E)$ be an essentially small exact category. We say that $(\A, \E)$ satisfies $(GR8)$ if for each indecomposable object $X$ the following holds:
if $X{\subset}_{\E}Y={{\oplus}^{r}_{i=1}}Y_i$ with indecomposables $Y_i$, then $\mu_{\E}(X)\lll \max_{1 \le i \le r}\mu_{\E}(Y_i)$, and $X$ is a direct summand of $Y$ in case equality holds.
\end{definition}

\begin{lemma}$(GR8)$ holds for the minimal exact structure $(\A, \E_{min})$.
\end{lemma}
\begin{proof}
If $X{\subset}_{\E}Y={{\oplus}^{r}_{i=1}}Y_i$ with respect to the minimal exact structure $\E=\E_{min}$, then $X$ is isomorphic to a direct summand $Y_j$. 
Thus $\mu_{\E}(X)\lll \max\mu_{\E}(Y_i)$, and $(GR8)$ holds.
\end{proof}

\begin{remark}\label{ds}
The main property $(GR8)$ holds for the maximal exact structure $\E_{ab}$ when $\A$ is abelian, and for the minimal exact structure $\E_{min}$. However, in general, if we have $X{\subset}_{\E}Y={{\oplus}^{r}_{i=1}}Y_i$, then  $\mu_{\E}(X) = \max\mu_{\E}(Y_i)$ does not always imply that $X$ is a direct summand of $Y$.
We provide an example:

Consider the example discussed in \ref{cube}, and choose the exact structure $\E=\E_3$. If we take $X = S_2$, then $X$  is an $\E-$subobject of $Y= P_1 \oplus P_3$ since we have the Auslander-Reiten sequence $(AR3)$ in $\E$. But all indecomposables are simple in $\E$, so the measure is $\mu_\E(X) = \mu_\E(P_1) = \mu_\E(P_3)=(1)$. That is, even if the condition $\mu_{\E}(X) = \max\mu_{\E}(Y_i)$ is satisfied, $X$ is \emph{not} a direct summand of $Y$.

This example also illustrates that the property $(GR8)$ is not preserved under reduction: It holds for $(\A,\E_{ab})$ and $(\A,\E_{min})$, but not for the intermediate exact category $(\A,\E_{3})$.
In general, we do not know which class of exact categories satisfies $(GR8)$.
\end{remark}

Let us close this section by the following proposition which modifies the definition of the extension map in \cite[3.4]{Kr11}:
\begin{proposition}The Gabriel-Roiter measure \ref{construction}
can be extended to a \emph{measure} defined for all objects in $\A$ (not only the indecomposable ones) as follows: 

\[\mu_{\E}:(Obj \A,{\subset}_{\E})  \rightarrow (\mathfrak{S}(\mathbb{N}), \lll)\]
\[X\longmapsto {\mu}_{\E}(X)= max_{X'{\subset}_{\mathcal{E}}X}\, \mu_{\E}(X')\]
where $X'\in ind\A$ runs through all the indecomposable subobjects of $X$.
\end{proposition}
\begin{proof}
Clearly, ${\mu}_{\E}$ is an inclusion-preserving map between the poset $Obj\A$ with the partial order of \ref{poset}, and $(\mathfrak{S}(\mathbb{N}), \lll)$ which is totally ordered as we have seen above. So ${\mu}_{\E}$ verifies the condition in \ref{Ameasure}.
\end{proof}
\begin{example}We revisit examples \ref{submonoid2} and \ref{measure example}.
Consider $X=K^6$, then ${\mu}_{\E}(K^6)=max\{{\mu}_{\E_{min}}(K^2),{\mu}_{\E_{min}}(K^3)\}=(1)$.
\end{example}


\section{Basic properties under reduction of exact structures}\label{section under reduction}

The aim of this section is to investigate how the basic notions like the $\E-$length and Gabriel-Roiter measure change under reduction of  exact structures.
\subsection{Reduction of the $\E-$length }
Here we prove that the $\E-$length of objects gets reduced when we reduce the corresponding exact structure:
\begin{lemma} If $\E$ and $\E'$ are exact structures on $\A$ such that $\E' \subseteq \E$, then $l_{\E'}(X) \leq l_{\E}(X)$ for all objects $X$ of $\A$.
\end{lemma}
\begin{proof}Let us consider a maximal chain of $\E'-$admissibles monics ending by $X$
\[ 0=X_0 \; \imono{i_1} X_1 \;\imono{i_2} \cdots \; \imono{i_{n-1}} \;  X_{n-1}\;\imono{i_n}X=X_n \]
where $l_{\E'}(X)=n$. Since $\E'\subseteq \E$, all these pairs $(i_j, d_j)$  will also be in $\E$. So the chain above is also a chain of $\E-$admissibles monics and therefore by definition $l_{\E}(X)\ge n$.
\end{proof}

Let us illustrate reduction of length by an example:

\begin{example}
By taking $Ex(\A)$ as in \ref{cube}, we re-consider the example \ref{length ex} and notice that the chain of reductions \[\mathcal{E}_{min}\subseteq \mathcal{E}_{1,3,5}\subseteq \mathcal{E}_{ab}\] gives us that \[l_{\mathcal{E}_{min}}(I_2)= 1 \; < \; l_{\mathcal{E}_{1,3,5}}(I_2)=2 \; < \; l_{\mathcal{E}_{ab}}(I_2)=3.\]
\end{example}

\subsection{The behavior of $\E-$Gabriel-Roiter measure }
We notice that the $\E-$Gabriel-Roiter measure changes in different manners once we reduce the corresponding exact structure. Contrarily to the $\E-$length function, the Gabriel-Roiter measure does not always get reduced; reducing the exact structure could reduce the corresponding Gabriel-Roiter measure and could also enlarge it.
We illustrate this situation by some examples:

\begin{example}
We consider $\A = \rep Q$ where 
\[ Q : \qquad \xymatrix{ 1   & 2 \ar[l]_{\alpha} \ar[r]^{\beta} & 3 }\]

and we consider the following non-split short exact sequences:
\begin{enumerate}

\item[(AR1)] \hspace{2.5cm}$\xymatrix{ 0 \ar[r] & 100 \ar[r]& 111 \ar[r] & 011  \ar[r] & 0}$
\item[(AR2)] \hspace{2.5cm}$ \xymatrix{ 0 \ar[r] & 001 \ar[r]& 111 \ar[r] & 110 \ar[r]  & 0 }$
\item[(AR3)] \hspace{2cm}$ \xymatrix{ 0 \ar[r] & 111 \ar[r] & 011\oplus 110 \ar[r] & 010 \ar[r]  & 0 }$
\item[(4)] \hspace{2.5cm}$ \xymatrix{ 0 \ar[r] & 001 \ar[r] & 011 \ar[r] & 010  \ar[r] & 0}$
\item[(5)] \hspace{2.5cm}$ \xymatrix{ 0 \ar[r] & 100 \ar[r]& 110 \ar[r] & 010  \ar[r] & 0}$
\end{enumerate}
\bigskip
We construct the exact structures in the same way as in \ref{cube}, and as mentioned in \ref{A3} the lattice of exact structure $Ex(\A)$ is a cube similar to \ref{cube}.
The following chains of reduction
\[\mathcal{E}_{min}\subseteq 
 \mathcal{E}_{1}\subseteq \mathcal{E}_{1,2}\subseteq \mathcal{E}_{ab}\] 
\[\mathcal{E}_{min}\subseteq 
 \mathcal{E}_{3}\subseteq \mathcal{E}_{1,3,5}\subseteq \mathcal{E}_{ab}\] 
give us for the indecomposable object with dimension vector $111$:
\[{\mu}_{\mathcal{E}_{min}}(111)= (1)\lll {\mu}_{\mathcal{E}_{1}}(111)= {\mu}_{\mathcal{E}_{1,2}}(111)= (1,2) \ggg {\mu}_{\mathcal{E}_{ab}}(111)= (1,3)\]
\[{\mu}_{\mathcal{E}_{min}}(111)= {\mu}_{\mathcal{E}_{3}}(111)= (1) \lll {\mu}_{\mathcal{E}_{1,3,5}}(111)= (1,2) \ggg {\mu}_{\mathcal{E}_{ab}}(111)= (1,3)\]
and for $110$
\[{\mu}_{\mathcal{E}_{min}}(110)= {\mu}_{\mathcal{E}_{3}}(110)= (1) \lll {\mu}_{\mathcal{E}_{1,3,5}}(110)={\mu}_{\mathcal{E}_{ab}}(110)= (1,2).\]
So we notice that for this fixed additive category $\A=\rep Q$, by reducing the exact structure $\E$, the corresponding $\E-$Gabriel-Roiter measure gets sometimes reduced, and other times enlarged, even for the same indecomposable object.
\end{example}


\end{document}